\documentclass[a4paper,11pt,en,bibtex]{elegantpaper}
\usepackage{color,extarrows}
\title{Liouville-type Theorems for Stable Solutions of the H\'{e}non--Lane--Emden System \thanks{This work is supported by  National Key R\&D Program of China (Grant 2023YFA1010001) and NSFC (12171265 and 12271184)}}
\author{Long-Han Huang \and Wenming Zou}
\date{}
\usepackage{esint} 
\usepackage{graphicx} 
\usepackage{float} 
\usepackage{subfigure} 
\usepackage{appendix}
\numberwithin{equation}{section}
\newcommand{\xqedhere}[2]{\rlap{\hbox to#1{\hfil\llap{\ensuremath{#2}}}}} 
\newcommand{\NN}{\ensuremath{\mathbb{N}}}

\newcommand{\RR}{\ensuremath{\mathbb{R}}}

\newcommand{\md}{\mathrm{\,d}}
\newcommand{\ra}{\rightarrow}
\newcommand{\oo}{\infty}

\newcommand{\f}{\forall\,}
\newcommand{\ol}{\overline}

\newcommand{\supp}{\operatorname{supp\,}}

\newcommand{\sub}{\subseteq}
\newcommand{\Sub}{\Subset}
\newcommand{\tl}{\tilde}
\newcommand{\pt}{\partial}

\newcommand\blfootnote[1]{%
	\begingroup
	\renewcommand\thefootnote{}\footnote{#1}%
	\addtocounter{footnote}{-1}%
	\endgroup
}

\theoremstyle{plain}
\newtheorem{conj}{Conjecture}

\begin{document}

\maketitle

\blfootnote{This is the accepted version of the following article: Long-Han Huang and Wenming Zou, Liouville-type theorems for stable solutions of the Hénon–Lane–Emden system, \textit{Journal of the London Mathematical Society}, which has been published in final form at \url{https://doi.org/10.1112/jlms.70412}.}

\vskip0.6in
\begin{abstract}
	\rm	We investigate the H\'{e}non--Lane--Emden system defined by
\begin{equation*}
	\left\{
	\begin{aligned}
		- \Delta u=|x|^a |v|^{p-1}v  \\
		- \Delta v=|x|^b |u|^{q-1}u
	\end{aligned}
	\quad \mbox{in}\ \  \RR^N \!\setminus\! \{0\}.
	\right.
\end{equation*}
We begin by establishing a general Liouville-type theorem for the subcritical case. Then we prove that the H\'{e}non--Lane--Emden conjecture is valid for solutions stable outside a compact set, provided that  \(0 < \min\,\{p, q\} < 1\), or \(0 \leq a - b \leq (N-2)(p - q)\), or \(N \leq \frac{2(p+q+2)}{pq-1} + 10\). Additional Liouville-type theorems for the subcritical case are also obtained. Furthermore, we address the supercritical case. To our knowledge, these results constitute the first Liouville-type theorems for this class of solutions in the H\'{e}non--Lane--Emden system. As a by-product, several existing results in the literature are refined.



\end{abstract}

\section{Introduction}

\noindent In this paper, we examine Liouville-type theorems, specifically the nonexistence of positive solutions, for the following H\'{e}non--Lane--Emden system:
\begin{equation}\label{HLE}
	\left\{
	\begin{aligned}
		- \Delta u=|x|^a |v|^{p-1}v  \\
		- \Delta v=|x|^b |u|^{q-1}u
	\end{aligned}
	\quad \mbox{in}\ \  \RR^N \!\setminus\! \{0\},
	\right.
\end{equation}
	where $N \in \NN_+,p,q>0,a,b \in \RR.$ This system, \eqref{HLE}, serves as a natural generalization of the classical Lane--Emden system:
\begin{equation}\label{LE}
	\left\{
	\begin{aligned}
		- \Delta u=|v|^{p-1}v  \\
		- \Delta v=|u|^{q-1}u
	\end{aligned}
	\quad \mbox{in}\ \  \RR^N,
	\right.
\end{equation}
	where $N \in \NN_+,p,q>0.$ Over the past three decades, this system has attracted considerable attention from mathematicians. A notable conjecture regarding the nonexistence of positive solutions for \eqref{LE} is as follows.

\begin{conj}[Lane--Emden conjecture]\label{A}
	Let $ N \in \NN_+$ and $p,q>0.$ If $(p,q) $ is subcritical, i.e.,
	\begin{equation*}
		\dfrac{N}{p+1}+\dfrac{N}{q+1}>N-2,
	\end{equation*}
	then \eqref{LE} admits no positive solution in $\left( C^2(\RR^N) \right)^2$.
\end{conj}

Up to now, P. Souplet {\cite[Thm.\,1 and Thm.\,2]{zbMATH05563881}} has achieved the most significant progress on Conjecture \ref{A}, providing a complete resolution for $N \le 4$ and substantial partial results for $N\ge 5$. However, the conjecture remains open in higher dimensions. For further partial results, see {\cite[p.\,1411]{zbMATH05563881}}.

	A similar conjecture applies to the H\'{e}non--Lane--Emden system, as outlined below.
	
\begin{conj}[H\'{e}non--Lane--Emden conjecture]\label{B}
	Let $ N \ge 2,p,q>0$ and $a,b \in \RR.$ If $(p,q,a,b) $ is subcritical, i.e.
	\begin{equation*}
		\dfrac{N+a}{p+1}+\dfrac{N+b}{q+1}>N-2,
	\end{equation*}
	then \eqref{HLE} admits no positive solutions in $\left(C^2(\mathbb{R}^N \!\setminus\! \{0\}) \cap C(\mathbb{R}^N)\right)^2$.
\end{conj}

	It is important to note that the conjecture is false for $N=1$ (cf. {\cite[Prop.\,A.1 and Rem.\,A.2]{zbMATH05998062}}). Prior to discussing Conjecture \ref{B}, we introduce several important parameters. For $pq \neq 1,$ we define
\begin{align*}
	\alpha=\dfrac{2(p+1)}{pq-1},\quad \beta=\dfrac{2(q+1)}{pq-1},\quad  \tilde{\alpha}=\dfrac{2(p+1)+a+pb}{pq-1} , \quad
	\tilde{\beta}=\dfrac{2(q+1)+b+qa}{pq-1}.
\end{align*}
	Be warned that the definitions of $\alpha$ and $\beta$ here are different from some other literature. These parameters represent scaling exponents due to the fact that: For every $R>0,$ if $(u,v)$ is a solution of \eqref{HLE}, then $(u_R,v_R) := (R^{\tilde{\alpha}} u(R\cdot), R^{\tilde{\beta}} v(R\cdot))$ is also a solution of \eqref{HLE}.

\vskip0.1in

Now we outline the progress made toward proving Conjecture \ref{B}. First, for radial solutions, M.-F. Bidaut-Veron and H. Giacomini established the conjecture in \cite[Thm.\,1.4]{zbMATH05907486}. Specifically, they showed that if \( N \geq 3 \), \( pq > 1 \), and \( a, b > -2 \), then \eqref{HLE} admits no radial positive solutions if and only if \( (p, q, a, b) \) is subcritical. Additionally, it is known that \eqref{HLE} does not admit positive supersolutions for \( N \geq 2 \) if any of the following conditions is satisfied:
	\[
	pq \leq 1, \quad \min \{a, b\} \leq -2, \quad \text{or} \quad \max \{\tilde{\alpha}, \tilde{\beta}\} \geq N - 2.
	\]
	This result has been verified by Phan \cite[Thm.\,A]{zbMATH06121773} and by Li--Zhang \cite[Thm.\,3.1]{zbMATH06973892} for \( N \geq 3 \), as well as by Cheng--Li--Zhang \cite[Thm.\,3.1]{zbMATH07522895} for \( N \geq 2 \). It is worth pointing out that this result also holds when \eqref{HLE} is considered on an exterior domain, as detailed in \cite[Sec.\,6]{zbMATH05992595}. Furthermore, this result implies Conjecture \ref{B} for \( N = 2 \), which follows from the observation that
	\[
	\tilde{\alpha} = \frac{2(p + 1) + a + pb}{pq - 1} > \frac{2(p + 1) - 2 - 2p}{pq - 1} = 0,
	\]
	under the conditions \( pq > 1 \) and \( \min \{a, b\} > -2 \).

\vskip0.13in
	
	The problem has been resolved for $N=3$ by K. Li and Z. Zhang \cite[Thm.\,1.1]{zbMATH06973892}. Notably, it had previously been confirmed for $N=3$ under additional assumptions, such as the boundedness of solutions, as shown by Phan \cite[Thm.\,1.1]{zbMATH06121773} and Fazly--Ghoussoub \cite[Thm.\,1]{zbMATH06265786}. Recently, H. Li \cite[Thm.\,1]{zbMATH07599863} extended the validity of the H\'{e}non--Lane--Emden conjecture for \( N = 4 \) when \( p, q < \frac{4}{3} \) and for \( N = 5 \) when \( p, q < \frac{10}{9} \), employing a strategy similar to that of \cite{zbMATH06121773}, \cite{zbMATH06265786}, and \cite{zbMATH06973892}.  By synthesizing the results from \cite{zbMATH06121773} which demonstrates that \eqref{HLE} admits no positive solutions if \( (p, q, a, b) \) is subcritical and if \eqref{LE} admits no positive solution  and from \cite[Thm.\,A]{zbMATH00711288}, one can encompass the findings of \cite{zbMATH07599863}.

\vskip0.12in
	On the other hand, in the past decade, the Liouville-type theorems for stable solutions of \eqref{LE} and other related systems have garnered significant attention from experts. Let us first recall the notion of stability, which is motivated by {\cite{zbMATH02190057}}, {\cite{zbMATH06203801}} and {\cite{dupaignegherguhajlaoui2025}}. Here and throughout, we set
	\[
	B_R = \left\{ x \in \mathbb{R}^N : |x| < R \right\}, \quad R > 0,
	\]
	and we define \((u, v) \in \left(C^2(\mathbb{R}^N \setminus \{0\}) \cap C(\mathbb{R}^N)\right)^2\) to be a solution of \eqref{HLE} if it satisfies \eqref{HLE} pointwise in \(\mathbb{R}^N \!\setminus\! \{0\}\).

\vskip0.1in
\begin{definition}\label{1.3}
	Let $N \in \NN_+,\Omega \sub \RR^N$ be an open set and $(f,g) \in \left( C^1(\Omega \times (0,+\oo)) \right)^{2} .$ A positive solution $(u,v) \in \left( C^2(\Omega) \right)^2$ of
	\begin{align}\label{eq1.4}
		\left\{\begin{aligned}
			-\Delta u &= f(x,v(x)) \\
			-\Delta v &= g(x,u(x))
		\end{aligned}
		\quad \mbox{in}\ \  \Omega
		\right.\end{align}
	is called stable (resp. stable in the weak sense) if there exist $0< \xi,\zeta \in  C^2(\Omega )  \, (\mbox{resp. } H_{\rm{loc}}^{2}(\Omega) \cap C(\Omega) ) $ such that
	\begin{align*}\left\{\begin{aligned}
			-\Delta \xi &\ge D_v f(x,v(x))\zeta \\
			-\Delta \zeta &\ge D_u g(x,u(x))\xi
		\end{aligned}
		\quad \mbox{a.e. in}\ \  \Omega.
		\right.\end{align*}
	A positive solution of \eqref{HLE} is called stable outside a compact set $ K \sub \RR^N$ (resp.\ stable outside a compact set $ K \sub \RR^N$ in the weak sense) if it is stable in every open set $ \Omega \Subset \RR^N \!\setminus\! K$ (resp. if there exist $R_{2} > R_{1} >0$ such that it is stable in $B_{R_{2}R} \!\setminus\! \ol{B_{R_{1}R}}$ in the weak sense for every sufficiently large $R>0$).
\end{definition}
	
	The initial Liouville-type theorem for stable solutions of the complete Lane--Emden system was likely established by Cowan {\cite[Thm.\,1]{zbMATH06203801}}, who proved that \eqref{LE} admits no stable positive solution when \( N \leq 10 \) and \( p, q \geq 2 \). This result was subsequently refined by Hajlaoui, Harrabi, and Mtiri {\cite[Thm.\,1.1]{zbMATH06672765}}, who demonstrated that \eqref{LE} admits no stable positive solution if \( N \leq 10 \) and \( p, q > \frac{4}{3} \), and that \eqref{LE} admits no bounded stable positive solution when \( N \leq 6 \) and \( p \geq q > 1 \). Recently, Dupaigne, Ghergu, and Hajlaoui {\cite[Thm.\,1.1]{dupaignegherguhajlaoui2025}} achieved the strongest result to date, proving that if \( N \leq 10 \) and \( p \geq q \geq 1 \) with \( pq \neq 1 \), then \eqref{LE} admits no stable positive solution, regardless of boundedness.

\vskip0.1in	
	Regarding the Liouville-type theorem for solutions of the Lane--Emden system that are stable outside a compact set, Mtiri and Ye \cite[Thm.\,1.1]{zbMATH07020411} confirmed that Conjecture \ref{A} is valid for such solutions. Moreover, Dupaigne, Ghergu, and Hajlaoui \cite[Thm.\,1.1]{dupaignegherguhajlaoui2025} further showed that if \( N \leq 10 \), \( p \geq q \geq 1 \) with \( pq \neq 1 \), and
	\begin{align}\label{eq1.5}
		\frac{N}{p+1} + \frac{N}{q+1} \neq N - 2,
	\end{align}
	then \eqref{LE} admits no positive solution in $\left( C^2(\mathbb{R}^N) \right)^{2}$ that is stable outside a compact set.

	We emphasize that the above restriction \eqref{eq1.5} is crucial. Indeed, P.-L. Lions {\cite[]{zbMATH04155283}} proved that if $N \ge 3,$ $ p,q>0$, and $\frac{N}{p+1} + \frac{N}{q+1} = N-2,$ then \eqref{LE} has a ground state, which is unique up to scalings and translations. By construction, such a solution is positive, radial and stable outside a compact set; see also {\cite[Thm.\,1]{zbMATH00936724}}.

\vskip0.1in

	Motivated by the preceding results, {\bf  two questions arise.} First, for solutions that are stable outside a compact set, does Conjecture \ref{B} hold? Second, what occurs if the quadruple \((p, q, a, b)\) is supercritical? To the best of our knowledge, these questions have not yet been explored. This gap in the literature may be attributed to the difficulties posed by the singularity of both the system and its solutions at the origin.

\vskip0.1in

	In this paper, we address these challenges by employing estimates of solutions on annular regions, thereby avoiding singularities at the origin. This approach enables us to establish new Liouville-type theorems for the system \eqref{HLE}. As a consequence, several of the aforementioned Liouville-type theorems will also be refined.

\subsection{Main results}\label{sec1.1}

	Our first main result establishes that the H\'{e}non--Lane--Emden conjecture is essentially valid for solutions that are stable outside a compact set. It should be stressed that solutions of \eqref{HLE} are always considered in $\left(C^2(\mathbb{R}^N \!\setminus\! \{0\}) \cap C(\mathbb{R}^N)\right)^2$. Parts (i) and (iii) of the following theorem address the cases of large and small values of $N$, respectively. Moreover, part (i) yields a sharper version of \cite[Thm.\,1.1]{zbMATH07020411}.
	
\begin{theorem}\label{1.4}
	Let $ N \ge 3, p,q>0, pq>1 $ and $a,b>-2.$ Assume $ (p,q,a,b)$ is subcritical. \\
	(i) If either $\min\,\{p,q\} < 1$ or
	\begin{align}\label{eq1.3}
		0 \le a - b \le (N - 2)(p - q),
	\end{align}
	then \eqref{HLE} admits no positive solution stable outside a compact set in the weak sense. \\
	(ii) The system \eqref{HLE} admits no positive solution $(u,v)$ satisfying
	\begin{align}\label{eq1.7}
		\begin{aligned}
			u=O(|x|^{-\tilde{\alpha}}),\quad v=O(|x|^{-\tilde{\beta}}), \quad  |x| \ra +\oo.
		\end{aligned}
	\end{align}
	(iii) If \eqref{LE} admits no bounded positive solution in $\left( C^{\oo}(\RR^N) \right)^{2}$ that is stable in every bounded open subset of $\RR^N$, then \eqref{HLE} admits no positive solution stable outside a compact set. In particular, the requirement is fulfilled if $(p,q)$ is subcritical, or if $N < 2\alpha z_{0} + 2, p \ge q>\frac{1}{3}$ and $ p> \frac{q+1}{3q-1}$, where $z_{0}$ is the largest root of the polynomial
	\[ G(x) = x^4-\frac{4 pq(q+1)}{p+1} x^2+\frac{2 pq(q+1)(p+q+2)}{(p+1)^2}x - \frac{ pq(q+1)^2}{(p+1)^2}. \]
\end{theorem}

\begin{remark}\label{1.5}
	If $N \ge 2,p,q>0,a,b \in \RR,$ then the nonexistence of positive (super) solutions of \eqref{HLE} implies the nonexistence of nonnegative nontrivial (super) solutions of \eqref{HLE}. Indeed, one can prove that if $(u,v)$ is a nonnegative nontrivial supersolution of \eqref{HLE}, then $u,v>0$ in $\RR^N.$ This is a easy corollary of the usual strong maximum principle and the maximum principle in punctured balls (see e.g. {\cite[Thm.\,1]{zbMATH03118705}}).
\end{remark}
	
	Theorem \ref{1.4} serves as an application of the following general result, which indicates that Liouville-type theorems for system \eqref{HLE} can be derived directly using energy estimates on annuli. Additionally, this result also shows that, to establish Liouville-type theorems for solutions of system \eqref{HLE} that are stable outside a compact set, it is unnecessary to first prove Liouville-type theorems for solutions that are stable in \(\mathbb{R}^N\). This approach differs from those utilized in \cite{zbMATH06203801}, \cite{zbMATH06672765}, \cite{zbMATH07020411}, and \cite{dupaignegherguhajlaoui2025}.

\begin{theorem}\label{1.6}
	Let $ N \ge 3, p,q>0, pq>1 $ and $a,b>-2.$ Assume $ (p,q,a,b)$ is subcritical and $(u,v)$ is a nonnegative solution of \eqref{HLE}. If there exist $R_{4}>R_{1}>0$ and $C_{0}>0$ such that
	\begin{align}\label{eq1.6}
		\begin{aligned}
			\int_{B_{R_{4}} \setminus B_{R_{1}}}  u_R^{q+1} +  v_R^{p+1} \le C_{0} , \quad \f R \gg 1,	
		\end{aligned}
	\end{align}
	then $u \equiv v \equiv 0$ in $\RR^N.$
\end{theorem}

	To evaluate the effectiveness of Theorem \ref{1.4}\,(iii), we provide the following estimate for the lower bound of $z_{0}$.

\begin{proposition}\label{1.7}
	Let $p, q > 0$, and let $z_0$ be as defined in Theorem \ref{1.4}\,(iii). \\
	(i) If $p > \frac{q+1}{3q-1}$ and $ q>\frac{1}{3},$ then $G<0$ in $ \big[ \frac{q+1}{2},z_{0} \big), pq>1$ and $z_{0}>\frac{q+1}{2} + \frac{1}{\alpha}.$ \\
	(ii) If $ p,q \ge 1$ and $pq \ne 1,$ then $p > \frac{q+1}{3q-1},q>\frac{1}{3}$ and $z_{0}>\frac{q+1}{2} + \frac{3}{\alpha}.$
\end{proposition}

	Proposition \ref{1.7} provides a refinement of the result of Dupaigne, Ghergu, and Hajlaoui {\cite[Thm.\,1.1 and Thm.\,1.7]{dupaignegherguhajlaoui2025}} concerning solutions that are stable in $\RR^N$, as illustrated by the following result. This is evident upon examining the original proof, where it is noted that the relation \( P = \alpha^{4} G\left(\frac{\cdot}{\alpha}\right) \) holds, with \( P \) as defined in {\cite[Thm.\,6]{dupaignegherguhajlaoui2025}}.

\begin{theorem}\label{1.8}
	Assume $N \in \NN_+, N<2\alpha z_{0}+2, p \ge q>\frac{1}{3}$ and $p > \frac{q+1}{3q-1} $, where $x_0$ is as defined in Theorem \ref{1.4}\,(iii). Then \eqref{LE} admits no positive solution in $\left( C^2(\RR^N) \right)^2$ that is stable in every bounded open subset of $\RR^N.$
\end{theorem}

	In alignment with Theorem \ref{1.6}, we also establish the following result. Recall that a function \( w \in C(\mathbb{R}^N) \) is referred to as polynomially bounded if there exist \( K, L > 0 \) such that
	\[
	|w(x)| \leq K |x|^L, \quad \f |x| \gg 1.
	\]
	
\begin{theorem}\label{1.9}
	Let $N \ge 3,p \ge q >0, pq>1$ and $a,b >-2.$ Assume $(p,q,a,b)$ is subcritical, and \eqref{eq1.3} holds provided $N \ge 4$. Additionally, suppose $(u,v)$ is a nonnegative solution of \eqref{HLE} that either is polynomially bounded or satisfies $|x|^{b}u^{q+1} + |x|^{a}v^{p+1} \in L^{1}(\RR^N)$. If there exist $R_{4}>R_{1}>0,s>0$ and $C_{0}>0$ such that $N-s\beta<1$ and
	\begin{equation}\label{eq1.8}
		\int_{B_{R_{4}} \setminus B_{R_{1}}} v_{R}^s \le C_{0}, \quad  \f R \gg 1,
	\end{equation}
	then $u \equiv v \equiv 0$ in $\RR^N.$ In particular, condition \eqref{eq1.8} holds if $s=p$ and $\alpha>N-3$.
\end{theorem}

	 It is worth pointing out that Theorem \ref{1.9} generalizes the results of Phan {\cite[Thm.\,1.1]{zbMATH06121773}}, Ze--Huang {\cite[Thm.\,1.1]{zbMATH07024017}}, and Souplet {\cite[Thm.\,1 and Thm.\,2]{zbMATH05563881}} (note Rem.\,1.1\,(a) therein). Our proof, however, adopts a different approach from the previous ones.
	
	 We now state our Liouville-type theorems for system \eqref{HLE} in the supercritical case. Notably, part (iii) strengthens the results of \cite[Thm.\,1.1 and Thm.\,1.7]{dupaignegherguhajlaoui2025} concerning solutions that are stable outside a compact set, when combined with Theorem \ref{1.4}\,(i) and Proposition \ref{1.7}.

\begin{theorem}\label{1.10}
	Let $N\ge 3, p\ge q>0, pq>1$ and $a,b>-2$. Assume $(p,q,a,b)$ is supercritical, i.e.
	\begin{equation*}
		\dfrac{N+a}{p+1}+\dfrac{N+b}{q+1} < N-2,
	\end{equation*}
	and \eqref{eq1.3} is satisfied. \\
		(i) Then \eqref{HLE} admits no positive solution $(u,v)$ satisfying $u \in L^{\frac{N}{\tilde{\alpha}}}(\RR^N)$ and
		\[ u=O(|x|^{-\tilde{\alpha}}),\quad |x| \ra +\oo.\]
		(ii) If \eqref{LE} admits no bounded positive solution in $\left( C^{\oo}(\RR^N) \right)^{2}$, then \eqref{HLE} admits no positive solution $(u,v)$ satisfying $u \in L^{\frac{N}{\tilde{\alpha}}}(\RR^N)$. Particularly, the requirement is satisfied if \((p, q)\) is subcritical, and either \(N \leq 4\) or \(\alpha > N - 3\). \\
		(iii) If \eqref{LE} admits no bounded positive solution in $\left( C^{\oo}(\RR^N) \right)^{2}$ that is stable in every bounded open subset of $\RR^N$, then \eqref{HLE} admits no positive solution $(u,v)$ that is both stable outside a compact set and satisfies $u \in L^{\frac{N}{\tilde{\alpha}}}(\RR^N)$. In particular, the requirement is satisfied if $(p,q)$ is subcritical, or if $N < 2\alpha z_{0} + 2, p \ge q>\frac{1}{3}$ and $ p> \frac{q+1}{3q-1}$, where $z_{0}$ is as defined in Theorem \ref{1.4}\,(iii).
\end{theorem}
\begin{remark}
	Building on the result established by Bidaut-Veron and Giacomini {\cite[Thm.\,1.4]{zbMATH05907486}}, the restriction \(u \in L^{\frac{N}{\tilde{\alpha}}}(\mathbb{R}^N)\) in Theorem \ref{1.10}\,(ii) is indispensable. For further conditions under which Theorem \ref{1.10}\,(ii) is applicable, see {\cite[pp. 1411--1412]{zbMATH05563881}}.
\end{remark}

\subsection{Structure of the paper}

	In Section \ref{sec.2}, we introduce the notations and conventions that will be used throughout the paper. Section \ref{sec.3} addresses the subcritical case, starting with the proof of Theorem \ref{1.6} and then applying it to derive Theorem \ref{1.4}. In the course of this discussion, we also prove Proposition \ref{1.7}, after which we move on to prove Theorem \ref{1.9}. In Section \ref{sec.4}, we turn our attention to the supercritical case, concluding with the proof of Theorem \ref{1.10}. Finally, Appendix \ref{AA} includes several basic inequalities used in this work.

\section{Notations and Conventions}\label{sec.2}
	For $p,q>0,a,b \in \RR,$ the pair $(p,q)$ is called subcritical (supercritical) if
	\[ \frac{N}{p+1} + \frac{N}{q+1} > (<)\, N-2,\]
	and the quadruple $(p,q,a,b)$ is called subcritical (supercritical) if
	\[\dfrac{N+a}{p+1}+\dfrac{N+b}{q+1} > (<)\, N-2.\]
	It is important to note that if $pq>1$, then $(p,q,a,b)$ is subcritical (supercritical) iff
	\begin{align}\label{eq2.1}
		\begin{aligned}
			\tilde{\alpha} + \tilde{\beta} > (<)\, N-2.
		\end{aligned}
	\end{align}
	If $pq \neq 1,$ then let us define
\begin{align*}
	\alpha=\dfrac{2(p+1)}{pq-1},\quad \beta=\dfrac{2(q+1)}{pq-1},\quad  \tilde{\alpha}=\dfrac{2(p+1)+a+pb}{pq-1} , \quad
	\tilde{\beta}=\dfrac{2(q+1)+b+qa}{pq-1}.
\end{align*}
	It is easily seen that
\begin{equation*}
	p\beta = \alpha +2,\quad
	q\alpha = \beta +2,\quad
	p\tilde{\beta} = \tilde{\alpha} + a + 2,\quad
	q\tilde{\alpha} = \tilde{\beta} + b + 2.
\end{equation*}

	We say that $(u,v) \in \left( C^2(\RR^N \!\setminus\! \{0\}) \cap C(\RR^N) \right)^2$ is a solution of $\eqref{HLE}$ if it satisfies \eqref{HLE} pointwise in $\RR^N \!\setminus\! \{0\}.$  For $R>0$ and $x \in \RR^N$, we write $(u_R(x),v_R(x))=(R^{\tilde{\alpha}} u(Rx), R^{\tilde{\beta}} v(Rx)) .$ A function $w \in C(\RR^N)$ is called polynomially bounded, if there exist $K,L >0$ such that
	\[|w(x)| \le K|x|^L, \quad \f |x| \gg 1.\]

	Let us write
	\[B_R=\left\{x \in \mathbb{R}^N: |x|<R\right\}, \quad  S_R=\left\{x \in \mathbb{R}^N: |x|=R\right\}, \quad R>0.\]
	For $w \in C(\RR^N \!\setminus\! \{0\}) ,(R,\theta) \in (0,+\oo) \times S_{1} $ and $k \in [1,+\oo],$ we shall write
	\begin{align*}
		\begin{aligned}
		 	w(R,\theta) &= w(R\theta), \quad \ol{w}(R) = \dfrac{1}{|S_1|}\int_{S_1}w(R,\cdot) = \fint_{S_R} w , \\
		 \|w(R)\|_{k} &= \|w(R,\cdot)\|_{L^{k}(S_1)} = R^{-\frac{N-1}{k}} \|w\|_{L^k(S_R)},
		\end{aligned}
	\end{align*}
	where $|S_1|$ denotes the $(N-1)$-dimensional Hausdorff measure of $S_1$.

	Let $M$ be a $C^{\oo}$ Riemannian manifold with or without boundary. The Sobolev space on $M$ will be denoted by $W^{j,k}(M),$ where $j \in \NN_+$ and $ k \in [1,+\oo] ;$ see {\cite[Def.\,2.3]{zbMATH01179494}} for the definition of Sobolev spaces. For $1 \le i \le j$ and $w \in W^{j,k}(M),$ we denote by $D_M^{\,i} w$ the $i$-th covariant derivative of $w.$

\section{The subcritical case}\label{sec.3}
	
\subsection{Proof of Theorem \ref{1.6}}

	We begin by introducing two integral estimates that are crucial for the proof of Theorem \ref{1.6}. Firstly, we present the following lemma, which shows that the integral of a function over an annulus can control its integral over a corresponding sphere.

\begin{lemma}\label{2.30}
	Let $N \in \NN_+$ and $R_{3}>R_{2}>0.$ Assume $m \in \NN_+, k_i \in  [1,+\oo)$ and $ w_i \in C(\RR^N \!\setminus\! \{0\}),$ where $i=1,2,\cdots,m.$ Then for every $R>0,$ there exists $\tilde{R} \in (R_{2}R,R_{3}R)$ such that
	\[ \|w_i(\tilde{R})\|_{k_i}^{k_i}  \le  R_{2}^{-N} R_{3} (R_{3}-R_{2})^{-1} (m+1) R^{-N} \|w_i\|_{L^{k_i}( B_{R_{3}R} \setminus B_{R_{2}R} )}^{k_i} , \quad \f i.\]
\end{lemma}
\begin{proof}
	For every $i=1,2,\cdots,m,$ define
	\[E_i= \left\{r \in (R_{2}R,R_{3}R): \|w_i(r)\|_{k_i}^{k_i}  >  R_{3} (R_{3}-R_{2})^{-1} (m+1) r^{-N} \|w_i\|_{L^{k_i}(B_{R_{3}R} \setminus B_{R_{2}R} )}^{k_i} \right\}.\]
	Then by the coarea formula, we have
	\[\begin{aligned}
		\|w_i\|_{L^{k_i}(B_{R_{3}R} \setminus B_{R_{2}R})}^{k_i}
		&= \int_{R_{2}R}^{R_{3}R} \|w_i(r)\|_{{k_i}}^{k_i} r^{N-1} \md r \\
		& \ge \int_{E_i} \|w_i(r)\|_{{k_i}}^{k_i} r^{N-1} \md r\\
		&\ge \int_{E_i} R_{3} (R_{3}-R_{2})^{-1} (m+1) r^{-N} \|w_i\|_{L^{k_i}(B_{RR_{3}} \setminus B_{RR_{2}} )}^{k_i} r^{N-1} \md r\\
		&\ge  (R_{3}-R_{2})^{-1} (m+1) R^{-1} |E_i| \|w_i\|_{L^{k_i}(B_{R_{3}R} \setminus B_{R_{2}R} )}^{k_i} .
	\end{aligned}\]
	Hence $ |E_i| \le (R_{3}-R_{2}) (m+1)^{-1} R $ and
	\[ \left|\bigcup\limits_{i=1}^m  E_i\right| \le \sum\limits_{i=1}^{m} |E_i| \le  m (m+1)^{-1} (R_{3}-R_{2}) R < (R_{3}-R_{2}) R = |(R_{2}R,R_{3}R)|.\]
	Consequently, one can find $\tilde{R} \in (R_{2}R,R_{3}R) \setminus  \bigcup_{i=1}^m  E_i .$
	This $\tilde{R}$ is exactly what we need.
\end{proof}

	Conversely, if $(u,v)$ is a nonnegative solution of \eqref{HLE}, then we can use its integral over a sphere to control its integral over the corresponding ball.

\begin{lemma}[Rellich--Pokhozhaev type inequality, see for example  {\cite{zbMATH06973892}}]\label{2.13}
	Suppose that $N \ge 3,
	p,q>0,$ $pq>1,a,b>-2 $
	and $(u, v)$ is a nonnegative solution of \eqref{HLE}. Then for every $R>0$ and $c_1,c_2 \in \mathbb{R}$ with $c_1+c_2=N-2,$ there holds
	\[
	\begin{aligned}
		& \left(\dfrac{N+b}{q+1}-c_1\right) \int_{B_R}|x|^b u^{q+1}
		+\left(\dfrac{N+a}{p+1}-c_2\right) \int_{B_R}|x|^a v^{p+1} \\
		\leq &\, \dfrac{R^{b+1}}{q+1} \int_{S_R} u^{q+1}
		+\dfrac{R^{a+1}}{p+1} \int_{S_R} v^{p+1}
		+\frac{c_1-c_2}{2} \int_{S_R}\left(\dfrac{\partial v}{\partial \nu}u-\dfrac{\partial u}{\partial \nu} v\right) .
	\end{aligned}
	\]
\end{lemma}
\begin{remark}
	Under the assumptions of Lemma \ref{2.13}, if moreover $(p,q,a,b)$ is subcritical, then one can find $c_1,c_2 \in \mathbb{R}$ with $c_1+c_2=N-2$ such that
	\[\dfrac{N+b}{q+1}-c_1>0,\quad \dfrac{N+a}{p+1}-c_2>0.\]
	Thus it is easy to see that there exists $C=C(N,p,q,a,b)>0$ such that
	\begin{equation}\label{eq2.7}
		\int_{B_R}|x|^b u^{q+1} + |x|^a v^{p+1}
		\le C\int_{S_R} R^{b+1}u^{q+1} + R^{a+1}v^{p+1} + |Du|v + u|Dv|.
	\end{equation}
\end{remark}

	We are now in a position to prove Theorem \ref{1.6}.
	
\begin{proof}[\bf{Proof of Theorem \ref{1.6}.}]
	Fix $R_{2},R_{3}>0$ such that $R_{4}>R_{3}>R_{2}>R_{1}.$ Then for every sufficiently large $R>0$, by Lemma \ref{2.30}, there exists $\tilde{R} \in (R_{2},R_{3})$ such that
\begin{align}\label{eq3.62}
	\begin{aligned}
		\int_{S_{\tilde{R}}}  u_R^{q+1} \le C \int_{B_{R_{3}} \setminus B_{R_{2}}} u_R^{q+1} \le C  , \quad \int_{S_{\tilde{R}}}  v_R^{p+1} \le C \int_{B_{R_{3}} \setminus B_{R_{2}}} v_R^{p+1} \le C ,
	\end{aligned}
\end{align}
	and
\begin{align}\label{eq3.63}
	\begin{aligned}
		\|Du_R\|_{L^{1+\frac{1}{p}}(S_{\tilde{R}}) } \le C \|Du_R\|_{L^{1+\frac{1}{p}}(B_{R_{3}} \setminus B_{R_{2}})}, \quad
		\|Dv_R\|_{L^{1+\frac{1}{q}}(S_{\tilde{R}}) } \le C \|Dv_R\|_{L^{1+\frac{1}{q}}(B_{R_{3}} \setminus B_{R_{2}})},
	\end{aligned}
\end{align}
	where $C=C(N,p,q,a,b,R_{1},R_{4},C_0)>0.$ Combining \eqref{eq3.63}, Lemma \ref{2.1}\,(ii) and Lemma \ref{2.2}\,(ii), we must have
\begin{align}\label{eq3.64}
	\begin{aligned}
		\|Du_R\|_{L^{1+\frac{1}{p}}(S_{\tilde{R}}) } &\le
		C \|D u_R\|_{L^{1+\frac{1}{p}}(B_{R_{3}} \setminus B_{R_{2}})} \\
		&\le
		C \left(
		\|\Delta u_R\|_{L^{1+\frac{1}{p}} (B_{R_{4}} \setminus B_{R_{1}})}
		+\|u_R\|_{L^1(B_{R_{4}} \setminus B_{R_{1}})}
		\right)  \\
		&\le C \left( \| |x|^a v_R^{p} \|_{L^{1+\frac{1}{p}} (B_{R_{4}} \setminus B_{R_{1}})} + 1 \right) \\
		&\le C .
	\end{aligned}
\end{align}
	Likewise, it is easily seen that
	\begin{align}\label{eq3.65}
		\begin{aligned}
			\|Dv_R\|_{L^{1+\frac{1}{q}}(S_{\tilde{R}})} \le C.
		\end{aligned}
	\end{align}
	Using H\"{o}lder's inequality and \eqref{eq3.62}, \eqref{eq3.64} and \eqref{eq3.65}, we deduce that
\begin{align}\label{eq3.66}
	\begin{aligned}
		\int_{S_{\tilde{R}}} |Du_R|v_R  &\le  \|Du_R\|_{L^{1+\frac{1}{p}}(S_{\tilde{R}})} \|v_R\|_{L^{p+1}(S_{\tilde{R}})} \le  C,
	\end{aligned}
\end{align}
	and
\begin{align}\label{eq3.67}
	\begin{aligned}
		\int_{S_{\tilde{R}}} u_R|Dv_R| &\le  \|u_R\|_{L^{q+1}(S_{\tilde{R}})} \|Dv_R\|_{L^{1+\frac{1}{q}}(S_{\tilde{R}})}  \le  C.
	\end{aligned}
\end{align}
	Now from \eqref{eq2.7}, \eqref{eq3.62}, \eqref{eq3.66} and \eqref{eq3.67}, we see that
\begin{align}\label{eq3.68}
	\begin{aligned}
		F_R(R_{1})&= \int_{B_{R_{1}}} |x|^b u_R^{q+1} + |x|^a v_R^{p+1} \\
		&\le \int_{B_{\tilde{R}}}|x|^b u_R^{q+1} + |x|^a v_R^{p+1} \\
		&\le C \int_{S_{\tilde{R}}} u_R^{q+1} + v_R^{p+1}  + |Du_R|v_R + u_R|Dv_R| \\
		& \le C .
	\end{aligned}
\end{align}
	It is easy to check that
	\begin{align}\label{eq3.69}
		\begin{aligned}
			F_R(r) = F_1(Rr) R^{\tilde{\alpha}+\tilde{\beta}-N+2}, \quad \f r>0.
		\end{aligned}
	\end{align}
	From this and \eqref{eq3.68}, it follows that
	\[R^{\tilde{\alpha}+\tilde{\beta}-N+2} F_1(R_{1}R) = F_R(R_{1}) \le C, \quad \f R \gg 1,\]
	which leads to $F_1(R_{1}R) \ra 0$ as $R \ra +\oo$ by \eqref{eq2.1}. This completes the proof.	
\end{proof}

\subsection{Applications of Theorem \ref{1.6}}

\subsubsection{Proof of Theorem \ref{1.4}\,(i): The case $0<\min\,\{p,q\}<1$}

	Following the ideas of Mtiri--Ye {\cite{zbMATH07020411}}, we consider a more general system.

\begin{definition}\label{3.9}
	Let $N \in \NN_+, 0 < q < 1 < q^{-1} < p$ and $a,b \in \RR$. Assume $ \Omega \sub \RR^N$ is a bounded open set with $0 \not\in \ol{\Omega}.$ We say that $v \in W_{\rm{loc}}^{2,\frac{1}{q} + 1}(\Omega) \cap L_{\rm{loc}}^{p+1}(\Omega)$ is a local weak supersolution of
	\begin{align}\label{eq3.25}
		\begin{aligned}
			\Delta\left( |x|^{-\frac{b}{q}} |\Delta v|^{\frac{1}{q}-1}\Delta v \right) =  |x|^a |v|^{p-1} v \quad \operatorname{in}\  \Omega,
		\end{aligned}
	\end{align}
	if
	\begin{align}\label{eq3.26}
		\begin{aligned}
			\int_{\Omega} |x|^{-\frac{b}{q}} |\Delta v|^{\frac{1}{q}-1} \Delta v \Delta \eta \le  \int_{\Omega} |x|^a |v|^{p-1} v \eta, \quad \f \eta \in C_{c}^{2}(\Omega).
		\end{aligned}
	\end{align}
	It is said to be stable if
	\begin{align}\label{eq3.27}
		\begin{aligned}
			\frac{1}{q} \int_{\Omega} |x|^{-\frac{b}{q}} |\Delta v|^{\frac{1}{q}-1}  |\Delta \eta|^{2} \ge p \int_{\Omega} |x|^a |v|^{p-1} \eta^{2}, \quad \f \eta \in C_{c}^{2}(\Omega).
		\end{aligned}
	\end{align}
\end{definition}

\begin{remark}
	It is not difficult to check that if $v \in W_{\rm{loc}}^{2,\frac{1}{q} + 1}(\Omega) \cap L_{\rm{loc}}^{p+1}(\Omega) $ is a local weak supersolution of \eqref{eq3.25} that is stable, then for every $\eta \in W^{2,\frac{1}{q} + 1}(\Omega) \cap L^{p+1}(\Omega)$ with $\supp (\eta) \sub \Omega,$ both \eqref{eq3.26} and \eqref{eq3.27} still hold. Indeed, this can be easily seen by a standard mollified argument and H\"{o}lder's inequality.
\end{remark}

	Next, we establish a connection between the notion of stability for systems and that for scalar equations.

\begin{lemma}[See {\cite[Proof of Lem.\,2.1]{zbMATH07020411}} or {\cite[Lem.\,2.1]{dupaignegherguhajlaoui2025}}]\label{3.10}
	Let $0<\zeta \in W_{\rm{loc}}^{2,k}(\Omega) \cap C(\Omega) $ for some $k \in [1,+\oo]$ and $\eta \in W_{\rm{loc}}^{2,\oo}(\Omega).$ Then $\frac{\eta^{2}}{\zeta} \in W_{\rm{loc}}^{2,k}(\Omega). $ If in addition that $-\Delta \zeta \ge 0$ a.e. in $\Omega$, then
	\[ \Delta \zeta \Delta \left( \frac{\eta^{2}}{\zeta} \right) \le |\Delta \eta|^{2} \quad \mbox{a.e. in}\ \  \Omega.\]
\end{lemma}

\begin{lemma}\label{3.14}
	Under the hypotheses of Definition \ref{3.9}, if moreover $(u,v) \in \left( C^{2}(\Omega) \right)^{2}$ is a positive solution of \eqref{HLE} in $\Omega$ that is stable in the weak sense, then $v \in C^{2}(\Omega)$ is a stable local weak supersolution of \eqref{eq3.25}.
\end{lemma}
\begin{proof}
	It is easily seen that $v \in C^{2}(\Omega)$ is a local weak supersolution of \eqref{eq3.25}. Now since $(u,v)$ is stable in the weak sense, there exist $0< (\xi,\zeta) \in \left( H_{\rm{loc}}^{2}(\Omega) \cap C(\Omega) \right)^{2} $ such that
	\begin{align*}
		\begin{aligned}\left\{\begin{aligned}
				-\Delta \xi &\ge p |x|^{a} v^{p-1} \zeta \\
				-\Delta \zeta &\ge q |x|^{b} u^{q-1} \xi
			\end{aligned}
			\quad \mbox{a.e. in}\ \  \Omega.
			\right.\end{aligned}\end{align*}
	Hence for every $\eta \in C_{c}^{2}(\Omega), $ by Lemma \ref{3.10} and the divergence theorem, we have
	\begin{align*}
		\begin{aligned}
			p \int_{\Omega} |x|^a |v|^{p-1} \eta^{2}	&\le   - \int_{\Omega} \frac{ \eta^{2}}{\zeta} \Delta \xi  \\
			&= -\int_{\Omega} \xi \Delta\left( \frac{ \eta^{2}}{\zeta} \right) \\
			&\le -\int_{\left\{ \Delta\big( \frac{\eta^{2}}{\zeta} \big) < 0 \right\}} \xi \Delta\left( \frac{ \eta^{2}}{\zeta} \right) \\
			&\le \frac{1}{q} \int_{\left\{ \Delta\big( \frac{\eta^{2}}{\zeta} \big) < 0 \right\}} |x|^{-b} u^{1-q} \Delta \zeta \Delta \left( \frac{ \eta^{2} }{\zeta} \right) \\
			&\le \frac{1}{q} \int_{\left\{ \Delta\big( \frac{\eta^{2}}{\zeta} \big) < 0 \right\}} |x|^{-b} u^{1-q} \left| \Delta  \eta \right|^{2}  \\
			&\le \frac{1}{q} \int_{\Omega} |x|^{-\frac{b}{q}} |\Delta v|^{\frac{1}{q}-1}  |\Delta \eta|^{2},
		\end{aligned}
	\end{align*}
	which implies $v$ is stable.
\end{proof}

	 We subsequently derive an important energy estimate for solutions of \eqref{eq3.25}, which implies \eqref{eq1.6} for solutions of \eqref{HLE}. This require the following interpolation inequality. Note that it does not actually require $v$ to satisfy \eqref{eq3.26} or \eqref{eq3.27}.

\begin{lemma}[See {\cite[Lem.\,2.3]{zbMATH07020411}}]\label{3.13}
	Let $N \in \NN_+$ and $\Omega \sub \RR^N$ be an open set. Assume $k \ge \frac{m}{2} >1$ and $\varepsilon>0.$ If $v \in W_{\rm{loc}}^{2,m}(\Omega)$ and $\psi \in C_{c}^{2}(\Omega) $ such that $0 \le \psi \le 1,$ then
	\[ \int_{\Omega} |D v|^{m} |D\psi|^{m} \psi^{4k-m} \le \varepsilon \int_{\Omega} |\Delta v|^{m} \psi^{4k} + C \int_{\Omega} |v|^{m} \left( |D\psi|^{2m} + |D^2\psi|^{m} \right) \psi^{4k-2m} ,\]
	where $C=C(N,m,k,\varepsilon)>0.$
\end{lemma}
	
\begin{lemma}\label{3.15}
	Under the hypotheses of Definition \ref{3.9}, if in addition that $v \in W_{\rm{loc}}^{2,\frac{1}{q} + 1}(\Omega) \cap L_{\rm{loc}}^{p+1}(\Omega)$ is a stable local weak supersolution of \eqref{eq3.25} and
	\begin{align}\label{eq3.35}
		\begin{aligned}
			k \ge \max\left\{ \frac{1}{q}+1, \frac{(p+1)(q+1)}{2(pq-1)} \right\},
		\end{aligned}
	\end{align}
	then for every $\psi \in C_{c}^{2}(\Omega)$ such that $0\le \psi \le 1,$ we have
	\[ \int_{\Omega} |x|^{-\frac{b}{q}} |\Delta v|^{\frac{1}{q}+1} \psi^{4k} + \int_{\Omega} |x|^{a} |v|^{p+1} \psi^{4k} \le C\int_{\Omega} \left( |D\psi|^{2\left( \frac{1}{q}+1 \right)} + |\Delta \psi|^{\frac{1}{q}+1} + |D^2\psi|^{\frac{1}{q}+1} \right)^{\frac{(p+1)q}{pq-1}}, \]
	where $C=C(N,q,b,k,\inf_{ \Omega}|\cdot|,\sup_{\Omega} |\cdot|)>0.$
\end{lemma}
\begin{proof} We divide the proof into 3 steps.

\noindent \emph{Step 1. Estimate of $\int_{\Omega} |x|^{a} |v|^{p+1} \psi^{4k}.$} To simplify notation, we write $m = \frac{1}{q} + 1.$ Taking $\eta = v \psi^{2k} $ in \eqref{eq3.27}, then for every $\varepsilon>0,$ we have
	\begin{align}\label{eq3.28}
		\begin{aligned}
			pq \int_{\Omega} |x|^{a} |v|^{p+1} \psi^{4k} &\le \int_{\Omega} |x|^{-\frac{b}{q}} |\Delta v|^{m-2} |\Delta (v\psi^{2k})|^{2} \\
			&= \int_{\Omega} |x|^{-\frac{b}{q}} |\Delta v|^{m-2} \left| \psi^{2k}\Delta v + 2 DvD(\psi^{2k}) + v\Delta(\psi^{2k}) \right|^{2} \\
			&\le  (1+\varepsilon) \int_{\Omega} |x|^{-\frac{b}{q}} |\Delta v|^{m} \psi^{4k} + C_{1} \varepsilon^{-1} K,
		\end{aligned}
	\end{align}
	where $C_{1} = C_{1}(q,b,k,\inf_{ \Omega}|\cdot|,\sup_{\Omega} |\cdot|)>0$ and
	\begin{align}\label{}
		\begin{aligned}
			K = \int_{\Omega} \left[ v^{2} |\Delta v|^{m-2} |\Delta (\psi^{2k})|^{2} + |Dv|^{2} |\Delta v|^{m-2} |D (\psi^{2k})|^{2} \right].
		\end{aligned}
	\end{align}
	By Young's inequality and Lemma \ref{3.13}, we conclude that
	\begin{align}\label{}
		\begin{aligned}
			\int_{\Omega} v^{2}|\Delta v|^{m-2} |\Delta (\psi^{2k})|^{2} &\le C_{2} \int_{\Omega}  v^{2} |\Delta v|^{m-2} \left( |D\psi|^{4} + |\Delta\psi|^{2} \right) \psi^{4k-4} \\
			&\le \varepsilon^{2} \int_{\Omega} |x|^{-\frac{b}{q}} |\Delta v|^{m} \psi^{4k} + C_{2} \int_{\Omega} v^{m} \left( |D\psi|^{4} + |\Delta\psi|^{2} \right)^{\frac{m}{2}} \psi^{4k-2m} \\
			&\le \varepsilon^{2} \int_{\Omega} |x|^{-\frac{b}{q}} |\Delta v|^{m} \psi^{4k} + C_{2} \int_{\Omega} v^{m} \left( |D\psi|^{2m} + |\Delta\psi|^{m} \right) \psi^{4k-2m} , \\
		\end{aligned}
	\end{align}
	and
	\begin{align}\label{eq3.29}
		\begin{aligned}
			\int_{\Omega} |Dv|^{2} |\Delta v|^{m-2} |D (\psi^{2k})|^{2}	&=  4k^{2} \int_{\Omega} |Dv|^{2} |\Delta v|^{m-2} |D \psi|^{2} \psi^{4k-2}  \\
			&\le \varepsilon^{2} \int_{\Omega} |x|^{-\frac{b}{q}} |\Delta v|^{m} \psi^{4k} + C_{1} \varepsilon^{-m+2} \int_{\Omega} |Dv|^{m} |D\psi|^{m} \psi^{4k-m} \\
			&\le \varepsilon^{2} \int_{\Omega} |x|^{-\frac{b}{q}} |\Delta v|^{m} \psi^{4k} + C_{2} \int_{\Omega} |v|^{m} \left( |D\psi|^{2m} + |D^2\psi|^{m} \right) \psi^{4k-2m}
		\end{aligned}
	\end{align}
	where $C_{2}=C_{2}(N,q,b,k,\inf_{ \Omega}|\cdot|,\sup_{\Omega} |\cdot| ,\varepsilon)>0.$ Combining \eqref{eq3.28}--\eqref{eq3.29}, we see that
	\begin{align}\label{eq3.30}
		\begin{aligned}
			pq \int_{\Omega} |x|^{a} |v|^{p+1} \psi^{4k}
			\le
			(1+C_{1}\varepsilon) \int_{\Omega} |x|^{-\frac{b}{q}} |\Delta v|^{m} \psi^{4k} + C_{2} L,
		\end{aligned}
	\end{align}
	where \[L = \int_{\Omega} |v|^{m} \left( |D\psi|^{2m} +|\Delta \psi|^{m} + |D^2\psi|^{m} \right) \psi^{4k-2m}.\]
	\emph{Step 2. Estimate of $\int_{\Omega} |x|^{-\frac{b}{q}} |\Delta v|^{m} \psi^{4k}.$} Taking $\eta = v\psi^{4k}$ in \eqref{eq3.26}, we get
	\[ \int_{\Omega} |x|^{-\frac{b}{q}} |\Delta v|^{m} \psi^{4k} \le
	\int_{\Omega} \left[ |x|^{a} |v|^{p+1} \psi^{4k} + C_{1} |v| |\Delta v|^{m-1} |\Delta (\psi^{4k})| + C_{1} |Dv| |\Delta v|^{m-1} |D(\psi^{4k})| \right]. \]
	Following a similar approach as in Step 1, by applying Young's inequality and Lemma \ref{3.13}, it is straightforward to deduce that
	\begin{align}\label{eq3.32}
		\begin{aligned}
			(1-C_{1}\varepsilon) \int_{\Omega} |x|^{-\frac{b}{q}} |\Delta v|^{m} \psi^{4k}	\le \int_{\Omega}  |x|^{a} |v|^{p+1} \psi^{4k} + C_{2} L .
		\end{aligned}
	\end{align}
	\emph{Step 3.} Choosing \(\varepsilon > 0\) sufficiently small, multiplying \eqref{eq3.32} by \(\frac{1 + 2C_{1}\varepsilon}{1 - C_{1}\varepsilon}\) on both sides, and adding the result to \eqref{eq3.30}, we obtain
	\begin{align}\label{eq3.33}
		\begin{aligned}
			C_{1}\varepsilon  \int_{\Omega} |x|^{-\frac{b}{q}} |\Delta v|^{m} \psi^{4k} 	
			+ \left( pq-\frac{1+2C_{1}\varepsilon}{1-C_{1}\varepsilon}  \right) \int_{\Omega} |x|^{a} |v|^{p+1} \psi^{4k}
			\le C_{2} L,
		\end{aligned}
	\end{align}
	where
	\begin{align}\label{eq3.34}
		\begin{aligned}
			L &= \int_{\Omega} |v|^{m} \left( |D\psi|^{2m} +|\Delta \psi|^{m} + |D^2\psi|^{m} \right) \psi^{4k-2m} \\
			&\le  C_{1} \left[ \int_{\Omega} |x|^{a} |v|^{p+1} \psi^{\frac{(p+1)(4k-2m)}{m}} \right]^{\frac{m}{p+1}} \left[ \int_{\Omega} \left( |D\psi|^{2m} +|\Delta \psi|^{m} + |D^2\psi|^{m} \right)^{\frac{p+1}{p+1-m}} \right]^{\frac{p+1-m}{p+1}} \\
			&\le C_{1} \left( \int_{\Omega} |x|^{a} |v|^{p+1} \psi^{4k} \right)^{\frac{m}{p+1}} \left[ \int_{\Omega} \left( |D\psi|^{2m} +|\Delta \psi|^{m} + |D^2\psi|^{m} \right)^{\frac{p+1}{p+1-m}} \right]^{\frac{p+1-m}{p+1}}
		\end{aligned}
	\end{align}
	due to H\"{o}lder's inequality, and \eqref{eq3.35}, i.e., $ 4k \le \frac{(p+1)(4k-2m)}{m}.$ Recall that $pq>1,$ hence fixing $\varepsilon>0$ sufficiently small and combining \eqref{eq3.33} with \eqref{eq3.34}, we get the desired conclusion.
\end{proof}

	Now we are ready to establish Theorem \ref{1.4}\,(i) for the case where $0 < q < 1 < p$.

\begin{theorem}\label{3.21}
	Let $ N \ge 3,
	0<q<1< q^{-1} <p $ and $a,b>-2 $. If $(p,q,a,b)$ is subcritical, then \eqref{HLE} admits no positive solution stable outside a compact set in the weak sense.
\end{theorem}
\begin{proof}
	Suppose $(u,v)$ is a positive solution of \eqref{HLE} and $R_{5}>R_{4}>R_{1}>R_{0}>0$ such that $(u,v)$ is stable in $B_{R_{5}R} \!\setminus\! \ol{B_{R_{0}R}}$ in the weak sense for sufficiently large $R>0$. Then it is easy to check that $(u_{R},v_{R})$ is a positive solution of \eqref{HLE} that is stable in $B_{R_{5}} \!\setminus\! \ol{B_{R_{0}}}$ in the weak sense for sufficiently large $R>0$. Fix $\psi \in C_{c}^{2}(B_{R_{5}} \!\setminus\! \ol{B_{R_{0}}})$ such that $0\le \psi \le 1$ and $\left. \psi\right|_{B_{R_{4}} \setminus B_{R_{1}}} = 1.$ Using Lemma \ref{3.14} and Lemma \ref{3.15}, we conclude that there exists $C>0$ such that
	\[ \int_{B_{R_{4}} \setminus B_{R_{1}}} u_{R}^{q+1} + v_{R}^{p+1} \le C, \quad \f R \gg 1.\]
	Then one can use Theorem \ref{1.6} to derive a contradiction.
\end{proof}

\subsubsection{Proof of Theorem \ref{1.4}\,(i): The case \eqref{eq1.3} is satisfied}

	In this case, to drive the energy estimate \eqref{eq1.6}, we will employ a strategy similar to that of Cowan {\cite{zbMATH06203801}}, Hajlaoui--Harrabi--Mtiri {\cite{zbMATH06672765}}, and Dupaigne--Ghergu--Hajlaoui {\cite{dupaignegherguhajlaoui2025}}. Nonetheless, several challenges remain to be addressed, primarily due to the singularity of both the system and its solutions at the origin. Initially, we have the following observation.
	
\begin{lemma}\label{3.12}
	Let $N \ge 3, p,q>0$ and $a,b \in \RR $. Assume $0 \not\in \Omega \sub B_{R}$ is an open set for some $R>0$ and $(u,v)$ is a positive solution of \eqref{HLE}. Then, for every $A>\frac{1}{2}$ and $\theta \in \left( 1,\frac{N}{N-2} \right)$, we have
	\begin{align}\label{eq3.77}
		\begin{aligned}
			\|u^{2A}\eta^2\|_{L^{\theta}(\Omega)} \le C R^{-\left( 1-\frac{1}{\theta} \right)N + 2} \int_{\Omega} \left[ |x|^{a} u^{2A-1} v^p \eta^2 + u^{2A} ( |D\eta|^2 + |\Delta(\eta^2)| ) \right] \!,  \f \eta \in C_c^2(\Omega),
		\end{aligned}
	\end{align}
	where $C=C(N,A)>0$.
\end{lemma}
\begin{remark}
	This lemma indicates that to obtain the energy estimate \eqref{eq1.6}, it is sufficient to estimate $\int_{\Omega} |x|^{a} u^{2A-1} v^p \eta^2$ for sufficiently large $A$. Indeed, since the term $\int_{\Omega} u^{2A} (|D\eta|^2 + |\Delta(\eta^2)|)$ can be controlled, to a certain degree, by $\| u^{2A} \eta^{2} \|_{L^{\theta}(\Omega)}$ through the application of interpolation inequalities.
\end{remark}
\begin{proof}
	By utilizing Lemma \ref{2.1}\,(i), we have
	\begin{align}\label{eq3.13}
		\begin{aligned}
			\|u^{2A}\eta^2\|_{L^{\theta}(\Omega)} \le C R^{-\left( 1-\frac{1}{\theta} \right)N + 2} \|\Delta (u^{2A}\eta^2)\|_{L^1(\Omega)}  .
		\end{aligned}
	\end{align}
	Hence it remains to estimate $\|\Delta (u^{2A}\eta^2)\|_{L^1(\Omega)}$. First, it is obvious that
	\begin{align}
		\begin{aligned}
			\|\Delta (u^{2A}\eta^2)\|_{L^1(\Omega)} &\le C  \int_{\Omega} \left[ |x|^{a} u^{ {2A} -1} v^p \eta^2  +  u^{ {2A} -2} |Du|^2 \eta^2  +  u^{2A} |\Delta (\eta^2)| + u^{2A-1} |Du||\eta D\eta|  \right] \\
			&\le C  \int_{\Omega} \left[  |x|^{a} u^{ {2A} -1} v^p \eta^2 +   u^{ {2A} -2} |Du|^2 \eta^2 + u^{2A} ( |D\eta|^2 + |\Delta(\eta^2)| ) \right],
		\end{aligned}
	\end{align}
	due to the AM--GM inequality. Upon multiplying both sides of $-\Delta u = |x|^a v^p$ by $u^{2A-1}\eta^2$ and applying the divergence theorem, we see that
	\begin{align*}
		\int_{\Omega} |x|^a u^{2A-1} v^p \eta^2
		&= -\int_{\Omega} u^{2A-1}\eta^2 \Delta u \\
		&= \int_{\Omega} Du D(u^{2A-1}\eta^2)  \\
		&=  (2A-1) \int_{\Omega}  u^{2A-2} |Du|^2 \eta^2  + \frac{1}{2A}\int_{\Omega} D(u^{2A})D(\eta^2) \\
		&= (2A-1) \int_{\Omega}  u^{2A-2} |Du|^2 \eta^2  - \frac{1}{2A} \int_{\Omega} u^{2A} \Delta(\eta^2).
	\end{align*}
	For this reason, we obtain
	\begin{equation}\label{eq3.15}
		\int_{\Omega}  u^{2A-2} |Du|^2 \eta^2  \le C \int_{\Omega} |x|^{a} u^{2A-1} v^p \eta^2 + C \int_{\Omega} u^{2A} |\Delta(\eta^2)| .
	\end{equation}
	Now by combining equations \eqref{eq3.13}--\eqref{eq3.15}, the desired inequality is obtained.
\end{proof}

	Our next objective is to estimate $\int_{\Omega} |x|^{a} u^{2A-1} v^p \eta^2$. Specifically, we shall prove that for some $A \ge \frac{q+1}{2}$, there exists $C=C(N,p,q,a,b,A)>0$ such that
	\begin{align}\label{eq3.78}
		\begin{aligned}
			\int_{\Omega} |x|^{a} u^{2A-1} v^p \eta^2 \le C\int_{\Omega} u^{2A} \left[  |x|^{-2} \eta^2 + |x|^{-1}|D(\eta^2)| +|D\eta|^2 + |\Delta(\eta^2)|  \right] \!, \ \f \eta \in C_c^2(\Omega),
		\end{aligned}
	\end{align}
	Upon establishing this inequality and combining it with \eqref{eq2.2}, the estimate \eqref{eq1.6} follows, as demonstrated in the subsequent lemma.

\begin{lemma}\label{3.11}
	Under the hypotheses of Lemma \ref{3.12}, assume additionally that $A > \frac{q}{2}$ and \eqref{eq3.78} holds. \\
	(i) Then for every $\psi \in C_{c}^{2}(\Omega)$ satisfying $0 \le \psi \le 1$ and
	\begin{align}
		\begin{aligned}
			m \ge \max\left\{ 1, \frac{3\theta A- q-A}{(\theta-1) q} \right\},
		\end{aligned}
	\end{align}
	we have
	\begin{align}\label{eq3.81}
		\begin{aligned}
			\int_{\Omega} u^{2\theta A} \psi^{2\theta m} \le C R^{\delta} \left[ \int_{\Omega} u^q \left( |x|^{-2}\psi^{2} +|x|^{-1}|D\psi| + |D\psi|^{2}+ \psi|\Delta\psi| \right)^{\frac{2\theta A-q}{2(\theta -1)A}} \right]^{\frac{2\theta A}{q}} ,
		\end{aligned}
	\end{align}
	where $C=C(N,p,q,a,b,A,m)>0$ and
	\[ \delta = \frac{N-2}{\theta -1}\left( \frac{N}{N-2} - \theta  \right)\left(  \frac{2\theta A}{q}-1 \right)>0 . \]
	(ii) If $\, \Omega= B_{R_{5}} \!\setminus\! \ol{B_{R_{0}}} \,$ for some $R_{5}>R_{4}>R_{1}>R_{0}>0$, $A \ge \frac{q+1}{2}$ and $(u,v)$ satisfies \eqref{eq2.2}, then
	\[ \int_{B_{R_{4}} \setminus B_{R_{1}}} u^{q+1} + v^{p+1} \le C, \]
	where $C=C(N,p,q,a,b,R_{0},R_{5})>0$.
\end{lemma}
\begin{proof}
	(i) Let $\eta= \psi^m,$ where $m \ge 1$ will be determined later. Then by \eqref{eq3.78}, we deduce
	\begin{align}\label{eq3.79}
		\begin{aligned}
			\| u^{2A} \psi^{2m} \|_{L^{\theta}(\Omega)}
			&\le C R^{-\left( 1-\frac{1}{\theta} \right)N + 2} \int_{\Omega} u^{2A} \left[  |x|^{-2} \eta^2 + |x|^{-1}|D(\eta^2)| +|D\eta|^2 + |\Delta(\eta^2)|  \right] \\
			&\le C R^{-\left( 1-\frac{1}{\theta} \right)N + 2} \int_{\Omega} u^{2A} \psi^{2m-2} \left( |x|^{-2}\psi^{2} +|x|^{-1}|D\psi| + |D\psi|^{2}+ \psi|\Delta\psi| \right) \\
			&:= C R^{-\left( 1-\frac{1}{\theta} \right)N + 2} \int_{\Omega} u^{2A} \psi^{2m-2} \Psi ,
		\end{aligned}
	\end{align}
	where $C=C(N,p,q,a,b,A,m)>0.$
	Since $A > \frac{q}{2}>0$ and $\theta>1$, we have $q < 2A < 2\theta A.$ Next, we define \(\lambda \in (0,1)\) such that \(2A = q \lambda + 2\theta A(1-\lambda)\), and select \(m \geq 1\) large enough to satisfy
	\[
	\frac{2m - 2 - 2\theta m(1 - \lambda)}{\lambda} \geq 1,
	\]
	which leads to the condition that \(m\) fulfills \eqref{eq3.13}. Note that both \(\lambda\) and \(m\) are well defined, with \(m\) depending only on \(q\), \(A\) and \(\theta\). Now using H\"{o}lder's inequality, we obtain
	\begin{align}\label{eq3.80}
		\begin{aligned}
			\int_{\Omega} u^{2A} \psi^{2m-2} \Psi
			&=  \int_{\Omega}  \left( u^{q} \Psi^{\frac{1}{\lambda}} \psi^{\frac{2m-2-2\theta m(1-\lambda)}{\lambda}} \right)^{\lambda} \left(u^{2\theta A} \psi^{2\theta m}\right)^{1-\lambda} \\
			&\le  \int_{\Omega}  \left( u^{q} \Psi^{\frac{1}{\lambda}} \right)^{\lambda} \left(u^{2\theta A} \psi^{2\theta m}\right)^{1-\lambda} \\
			&\le  \left( \int_{\Omega} u^q \Psi^{\frac{1}{\lambda}} \right)^{\lambda} \left(\int_{\Omega} u^{2\theta A} \psi^{2\theta m}\right)^{1-\lambda}.
		\end{aligned}
	\end{align}
	Consequently, the conclusion follows easily from \eqref{eq3.79} and \eqref{eq3.80}. \\
	(ii) Fix $\psi \in C_{c}^{2}(B_{R_{5}} \!\setminus\! \ol{B_{R_{0}}})$ such that $0\le \psi \le 1$ and $\left. \psi \right|_{B_{R_{4}} \setminus B_{R_{1}}} = 1.$ It follows from \eqref{eq3.81} and Lemma \ref{2.2}\,(i) that
	\[ \int_{B_{R_{4}} \setminus B_{R_{1}}} u^{2\theta A} \le C \left( \int_{B_{R_{5}} \setminus B_{R_{0}}} u^{q} \right)^{\frac{2\theta A}{q}} \le C. \]
	Now since $2\theta A > q + 1$, applying \eqref{eq2.2} and H\"{o}lder's inequality, we obtain
	\[ \int_{B_{R_{4}} \setminus B_{R_{1}}} v^{p+1} \le C\int_{B_{R_{4}} \setminus B_{R_{1}}} u^{q+1} \le C,\]
	which concludes the proof.
\end{proof}

	To prove \eqref{eq3.78}, we first utilize the stability of the solutions to deduce a key estimate, namely \eqref{eq3.36}, as shown below.
	
\begin{lemma}\label{3.1}
		Let $N \in \NN_+, \Omega \sub \RR^N$ be an open set and $0<w \in W_{\rm{loc}}^{2,1}(\Omega) \cap C(\Omega),$ then
		\[\int_{\Omega} \frac{-\Delta w}{w} \varphi^2 \le \int_{\Omega} |D\varphi|^2, \quad
		\f \varphi \in C_c^2(\Omega).\]
\end{lemma}
\begin{proof}
		Assume first $w \in C^{2}(\Omega).$ Then applying the divergence theorem three times, we have
		\begin{align*}
			\int_{\Omega} \frac{-\Delta w}{w} \varphi^2 &= -\int_{\Omega} w \Delta\left(\frac{\varphi^2}{w}\right) \\
			&= \int_{\Omega} \left( \frac{\Delta w}{w} \varphi^2 - 2\frac{|Dw|^2}{w^2} \varphi^2 + 4\frac{Dw}{w}\varphi D\varphi - \Delta(\varphi^{2}) \right) \\
			&=\int_{\Omega} \left( -DwD\left(\frac{\varphi^2}{w}\right) - 2\frac{|Dw|^2}{w^2} \varphi^2 + 4\frac{Dw}{w}\varphi D\varphi \right) \\
			&=\int_{\Omega} \left( -\frac{|Dw|^2}{w^2} \varphi^2 + 2\frac{Dw}{w}\varphi D\varphi \right) \\
			&=\int_{\Omega} \left( -\left| \frac{Dw}{w} \varphi +D\varphi \right|^2 + |D\varphi|^2 \right) \\
			&\le \int_{\Omega} |D\varphi|^2.
		\end{align*}
		Now if $w \in W_{\rm{loc}}^{2,1}(\Omega) \cap C(\Omega),$ then let us denote by $w_{\varepsilon} \in C^{\oo}(\Omega_{\varepsilon})$ the standard mollification of $w,$ where $\varepsilon>0$ and
		\[ \Omega_{\varepsilon} = \{ x \in \Omega: \operatorname{dist\,}(x,\pt{\Omega}) > \varepsilon \} . \]
		Take an open set $ \Omega'$ such that $\supp (\varphi) \Subset \Omega' \Subset \Omega.$ It is well known that $w_{\varepsilon} \ra w$ in $W^{2,1}(\Omega')$ and $w_{\varepsilon} \ra w$ uniformly on $\Omega'.$ Note that $w>0,$ hence there exists $C>0$ such that $w_{\varepsilon} \ge C$ in $\Omega'$ provided $\varepsilon>0$ is sufficiently small. Consequently, we deduce
		\[ \int_{\Omega} \frac{-\Delta w}{w} \varphi^{2} =  \int_{\Omega'} \frac{-\Delta w}{w} \varphi^{2} = \lim\limits_{\varepsilon \ra 0^{+}} \int_{\Omega'} \frac{-\Delta w_{\varepsilon}}{w_{\varepsilon}} \varphi^{2} \le \int_{\Omega'} |D\varphi|^{2} = \int_{\Omega} |D\varphi|^2. \qedhere \]
\end{proof}

\begin{lemma}\label{3.4}
	Let $N \in \NN_+, \Omega \sub \RR^N$ be an open set and $ (f,g) \in \left( C^1(\Omega \times (0,+\oo)) \right)^{2} .$ Assume $(u,v) \in \left( C^{2}(\Omega) \right)^{2}$ is a positive solution of \eqref{eq1.4} which  is stable in the weak sense. Then
	\[ \int_{\Omega} \left[ D_{v} f(x,v) D_{u} g(x,u) \right]^{\frac{1}{2}} \varphi^{2} \le \int_{\Omega} |D\varphi|^{2}, \quad \f \varphi \in C_c^2(\Omega). \]
	In particular, if $f=|x|^{a} v^{p}$ and $g=|x|^{b} u^{q},$ then
	\begin{align}\label{eq3.36}
		\begin{aligned}
			\sqrt{pq} \int_{\Omega} |x|^{\frac{a+b}{2}} u^{\frac{q-1}{2}} v^{\frac{p-1}{2}} {\varphi}^2
			&\le \int_{\Omega} |D\varphi|^2 , \quad \f \varphi \in C_c^2(\Omega).
		\end{aligned}
	\end{align}
\end{lemma}
\begin{proof}
	By definition, there exists $0< (\xi,\zeta) \in  \left( H_{\rm{loc}}^{2}(\Omega) \cap C(\Omega) \right)^{2} $ such that
	\begin{align*}\left\{\begin{aligned}
			-\Delta \xi &\ge D_v f(x,v(x))\zeta \\
			-\Delta \zeta &\ge D_u g(x,u(x))\xi
		\end{aligned}
		\quad \mbox{a.e. in}\ \  \Omega.
		\right.
	\end{align*}
	From this and Lemma \ref{3.1}, we deduce
	\begin{align*}
		\begin{aligned}
		\int_{\Omega} \left[ D_{v} f(x,v) D_{u} g(x,u) \right]^{\frac{1}{2}} \varphi^{2}	&\le \frac{1}{2}\int_{\Omega} \left( D_{v} f(x,v) \frac{\zeta}{\xi} + D_{u} g(x,u) \frac{\xi}{\zeta} \right) \varphi^{2} \\
		&\le  \frac{1}{2}  \int_{\Omega}  \left(  \frac{-\Delta \xi}{\xi} \varphi^2 + \frac{-\Delta \zeta}{\zeta} \varphi^2\right) \\
		&\le \int_{\Omega} |D\varphi|^{2} ,
		\end{aligned}
	\end{align*}
	and the proof is complete.
\end{proof}
	
	Guided by this lemma, we shall focus our forthcoming analysis on positive solutions of \eqref{HLE} satisfying \eqref{eq3.36}. Subsequently, we will establish \eqref{eq3.78} for this class of solutions.
	
\begin{lemma}\label{3.16}
		Let $N \in \NN_+, p \ge q>0$ and $ a,b \in \RR $. Assume $\Omega \sub \RR^N$ is an open set with $0 \not\in \Omega$ and $(u,v)$ is a positive solution of \eqref{HLE} satisfying \eqref{eq3.36} in $\Omega .$ Additionally, assume $A> \max\left\{\frac{1}{2},\frac{q+1}{4}\right\}$ verifying $G(A)>0$ and $B=\frac{p+1}{q+1}A$, where $G$ is defined as Theorem \ref{1.4}\,(iii). \\
		(i) There exists $C=C(p,q,a,b,A)>0$ such that
		\begin{align}\label{eq3.60}
			\begin{aligned}
				\int_{\Omega} |x|^{\frac{a+b}{2}} u^{\frac{p-1}{2}} v^{\frac{q-1}{2}} u^{2A} \eta^2 \le CM, \quad \f \eta \in C_c^2(\Omega),
			\end{aligned}
		\end{align}
		where
		\begin{align}\label{eq3.59}
			\begin{aligned}
				M= \int_{\Omega} u^{2A}\left[ |D\eta|^2 + |\Delta(\eta^2)| \right] + \int_{\Omega} |x|^{\frac{2B(a-b)}{p+1}} v^{2B}\left[ |x|^{-2} \eta^2 + |x|^{-1}|D(\eta^2)| + |\Delta (\eta^2)| \right].
			\end{aligned}
		\end{align}
		(ii) If $(u,v)$ satisfies \eqref{eq2.2}, then \eqref{eq3.78} holds.
\end{lemma}
\begin{proof}
		(i) \emph{Step 1.} For every $\eta \in C_c^2(\Omega),$ by applying \eqref{eq3.36}, we see that
		\begin{align}\label{eq3.1}
			\begin{aligned}
				\sqrt{pq} \int_{\Omega} |x|^{\frac{a+b}{2}} u^{\frac{q-1}{2}} v^{\frac{p-1}{2}} u^{2A} \eta^2
				&\le \int_{\Omega} |D(u^A \eta)|^2    \\
				&=\int_{\Omega} \left[ |D(u^A)|^2 \eta^2 + 2Au^{2A-1} \eta Du D\eta + u^{2A}|D\eta|^2 \right].
			\end{aligned}
		\end{align}
		In that follows, we shall estimate each term on the right-hand side. According to the divergence theorem, we obtain
	\begin{align}\label{eq3.2}
		\begin{aligned}
			\int_{\Omega} |D(u^A)|^2 \eta^2 &= \frac{A^2}{2A-1} \int_{\Omega} \eta^{2} Du D(u^{2A-1})     \\
			&= -\frac{A^2}{2A-1} \int_{\Omega} u^{2A-1} \operatorname{div\,} (\eta^2Du)    \\
			&= -\frac{A^2}{2A-1} \int_{\Omega} u^{2A-1} \left[ \eta^2 \Delta u + Du D(\eta^2) 	\right]    \\
			&= \frac{A^2}{2A-1} \int_{\Omega} |x|^a u^{2A-1} v^p \eta^2 - \frac{A}{2(2A-1)} \int_{\Omega} D(u^{2A}) D(\eta^2)    \\
			&= \frac{A^2}{2A-1} \int_{\Omega} |x|^a u^{2A-1} v^p \eta^2 + \frac{A}{2(2A-1)} \int_{\Omega} u^{2A} \Delta(\eta^2),
		\end{aligned}
	\end{align}
		and
		\begin{align}\label{eq3.3}
			\int_{\Omega} 2Au^{2A-1} \eta Du D\eta = \frac{1}{2} \int_{\Omega}  D(u^{2A}) D(\eta^2) = -\frac{1}{2} \int_{\Omega} u^{2A} \Delta(\eta^2) .
		\end{align}
		Combining \eqref{eq3.1}--\eqref{eq3.3}, we have
		\begin{align}\label{eq3.4}
			\begin{aligned}
				I_1 := \int_{\Omega} |x|^{\frac{a+b}{2}} u^{\frac{q-1}{2}} v^{\frac{p-1}{2}} u^{2A} \eta^2
				\le A_1^{-1} \int_{\Omega} |x|^a u^{2A-1} v^p \eta^2 + CM,
			\end{aligned}
		\end{align}
		where $A_1=\frac{\sqrt{pq} \left(2A-1\right)}{A^2}$ and $C=C(p,q,A)>0$. \\
		\emph{Step 2.} Let us set $\tilde{v}=|x|^{\frac{a-b}{p+1}} v \in C^2(\Omega).$ Then, taking $\varphi={\tl{v}}^B\eta$ in \eqref{eq3.36}, we obtain
		\begin{align}\label{eq3.5}
			\begin{aligned}
				\sqrt{pq} \int_{\Omega} |x|^{\frac{a+b}{2}} u^{\frac{q-1}{2}} v^{\frac{p-1}{2}} \tilde{v}^{2B}\eta^2
				&\le \int_{\Omega} |D(\tilde{v}^B \eta)|^2    \\
				&=\int_{\Omega} \left[ |D(\tilde{v}^B)|^2 \eta^2 + 2B{\tilde{v}}^{2B-1} \eta D{\tilde{v}} D\eta + {\tilde{v}}^{2B}|D\eta|^2 \right].
			\end{aligned}
		\end{align}
		Similar to the calculations in \eqref{eq3.2}, we have
		\begin{align}\label{eq3.6}
			\begin{aligned}
				\int_{\Omega} |D(\tl{v}^B)|^2 \eta^2 &=  -\frac{B^2}{2B-1} \int_{\Omega}  {\tl{v}}^{2B-1} \eta^2 \Delta{\tl{v}} + \frac{B}{2(2B-1)} \int_{\Omega} {\tl{v}}^{2B} \Delta(\eta^2)  \\
				&= \frac{B^2}{2B-1} \int_{\Omega}  |x|^{\frac{a-b}{p+1} + b} u^q {\tl{v}}^{2B-1} \eta^2  + \frac{B}{2(2B-1)} \int_{\Omega} |x|^{\frac{2B(a-b)}{p+1}} v^{2B} \Delta(\eta^2) + J,
			\end{aligned}
		\end{align}
		where \begin{align}\label{eq3.7}
			\begin{aligned}
				J &= -\frac{B^2}{2B-1} \int_{\Omega} \left[ v\Delta\left( |x|^{\frac{a-b}{p+1}} \right) + 2D\left( |x|^{\frac{a-b}{p+1}} \right) Dv \right] |x|^{\frac{(2B-1)(a-b)}{p+1}} v^{2B-1} \eta^{2} \\
				&\le C \int_{\Omega} |x|^{\frac{2B(a-b)}{p+1}-2} v^{2B} \eta^{2} - \frac{1}{2(2B-1)} \int_{\Omega} \eta^{2} D\left( |x|^{\frac{2B(a-b)}{p+1}} \right) D(v^{2B}) \\
				&= C \int_{\Omega} |x|^{\frac{2B(a-b)}{p+1}-2} v^{2B} \eta^{2} + \frac{1}{2(2B-1)} \int_{\Omega} v^{2B} \operatorname{div} \left[ \eta^{2} D\left( |x|^{\frac{2B(a-b)}{p+1}} \right)  \right] \\
				&\le C\int_{\Omega} |x|^{\frac{2B(a-b)}{p+1}} v^{2B} \left[ |x|^{-2} \eta^{2} + |x|^{-1} |D(\eta^{2})| \right]
			\end{aligned}
		\end{align}
		for some $C=C(p,a,b,B)>0.$ Moreover, applying the divergence theorem again, we have
		\begin{equation}\label{eq3.61}
			2B\int_{\Omega}	{\tilde{v}}^{2B-1} \eta D{\tilde{v}} D\eta = \frac{1}{2} \int_{\Omega}  D(\tl{v}^{2B}) D(\eta^2) = -\frac{1}{2} \int_{\Omega} |x|^{\frac{2B(a-b)}{p+1}} {v}^{2B} \Delta(\eta^2).
		\end{equation}
		Now using \eqref{eq3.5}--\eqref{eq3.61}, we conclude that
		\begin{align}\label{eq3.8}
			\begin{aligned}
				I_2 := \int_{\Omega} |x|^{\frac{a+b}{2}} u^{\frac{q-1}{2}} v^{\frac{p-1}{2}} \tilde{v}^{2B}\eta^2
				\le  B_1^{-1}\int_{\Omega}  |x|^{\frac{a-b}{p+1} + b} u^q {\tl{v}}^{2B-1} \eta^2 + CM ,
			\end{aligned}
		\end{align}
		where $B_1=\frac{\sqrt{pq} \left(2B-1\right)}{B^2}$ and $C=C(p,q,a,b,B)>0$. \\
		\emph{Step 3.} Let $m=\frac{q+1}{4A}=\frac{p+1}{4B} \in (0,1),$ then by \eqref{eq3.4} and \eqref{eq3.8}, we have
		\begin{align}\label{eq3.9}
			\begin{aligned}
				A_1^{\frac{1}{m}} I_1 + I_2 &= A_1^{\frac{1}{m}}\int_{\Omega} |x|^{\frac{a+b}{2}} u^{\frac{q-1}{2}} v^{\frac{p-1}{2}} u^{2A} \eta^2 + \int_{\Omega} |x|^{\frac{a+b}{2}} u^{\frac{q-1}{2}} v^{\frac{p-1}{2}} \tilde{v}^{2B}\eta^2  \\
				&\le A^{\frac{1-m}{m}} \int_{\Omega} |x|^a u^{2A-1} v^p \eta^2 + B_1^{-1}\int_{\Omega}  |x|^{\frac{a-b}{p+1} + b} u^q {\tl{v}}^{2B-1} \eta^2 + CM.
			\end{aligned}
		\end{align}
		By Young's inequality, we must have
		\begin{align}\label{eq3.11}
			\begin{aligned}
				A^{\frac{1-m}{m}} \int_{\Omega} |x|^a u^{2A-1} v^p \eta^2
				&= \int_{\Omega} u^{\frac{q-1}{2}} v^{\frac{p-1}{2}} \left( A_1^{\frac{1}{m}} |x|^{\frac{a+b}{2}} u^{2A}  \right)^{1-m}
				\left(|x|^{\frac{a+b}{2}} \tl{v}^{2B} \right)^m \eta^2 \\
				&\le (1-m)A_1^{\frac{1}{m}} I_1 + mI_2,
			\end{aligned}
		\end{align}
		and
		\begin{align}\label{eq3.12}
			\begin{aligned}
				B_1^{-1}\int_{\Omega}  |x|^{\frac{a-b}{p+1} + b} u^q {\tl{v}}^{2B-1} \eta^2
				&= \int_{\Omega} u^{\frac{q-1}{2}} v^{\frac{p-1}{2}} \left( B_1^{-\frac{1}{m}} |x|^{\frac{a+b}{2}} u^{2A} \right)^m \left( |x|^{\frac{a+b}{2}} \tl{v}^{2B} \right)^{1-m} \eta^2 \\
				&\le mB_1^{-\frac{1}{m}} I_1 + (1-m) I_2.
			\end{aligned}
		\end{align}
		From \eqref{eq3.9}--\eqref{eq3.12}, it follows that
		\[A_1^{\frac{1}{m}} I_1 + I_2 \le (1-m)A_1^{\frac{1}{m}} I_1 + mI_2 + mB_1^{-\frac{1}{m}} I_1 + (1-m) I_2 + CM.\]
		This gives
		\[m\left[(A_1B_1)^{\frac{1}{m}} -1 \right] I_1 \le CM,\]
		where $C=C(p,q,a,b,A)>0.$ Finally, observing that $A_1B_1>1$ is equivalent to $G(A)>0$, the desired conclusion follows. \\
		(ii) In this case, by \eqref{eq2.2}, we have
		\[C|x|^{\frac{a-b}{2}} \le u^{\frac{q+1}{2}} v^{-\frac{p+1}{2}}  , \quad |x|^{\frac{2B(a-b)}{p+1}} v^{2B} \le Cu^{2A}.\]
		It follows that
		\[|x|^{\frac{a+b}{2}} u^{\frac{q-1}{2}} v^{\frac{p-1}{2}} u^{2A} = |x|^{\frac{a+b}{2}}  u^{\frac{q+1}{2}} v^{-\frac{p+1}{2}} u^{2A-1} v^p \ge C |x|^{a} u^{2A-1} v^p.\]
		Combining these with (i), we complete the proof.
\end{proof}

	Ultimately, since
	\[
	G\left( \frac{q+1}{2} \right) = - \frac{(q+1)^2(5pq+p+q+1)[(3q-1)p-q-1]}{16(p+1)^2} < 0,
	\]
	provided that $ p > \frac{q+1}{3q-1} $ and $ q > \frac{1}{3} $, we can readily derive the following Liouville-type theorem from Lemmas \ref{3.11} and \ref{3.16}. This theorem, together with Lemma \ref{2.3}, Lemma \ref{3.4} and Theorem \ref{3.21}, completes the proof of Theorem \ref{1.4}\,(i).

\begin{theorem}\label{3.22}
	Let $N \ge 3, p \ge q> \frac{1}{3}, p>\frac{q+1}{3q-1}$ and $a,b >-2$. Assume that $(p,q,a,b)$ is subcritical and $(u,v)$ is a positive solution of \eqref{HLE} satisfying \eqref{eq2.2}. If there exist $R_{5}>R_{0}>0$ such that $(u,v)$ satisfies \eqref{eq3.36} in $ B_{RR_{5}} \!\setminus\! \ol{B_{RR_{0}}} \,$ for every sufficiently large $R>0$, then $u \equiv v \equiv 0$ in $\RR^N.$
\end{theorem}

\subsubsection{Proof of (ii) and (iii) in Theorem \ref{1.4}}

\begin{proof}[\bf{Proof of Theorem \ref{1.4}\,(ii).}]
	Let us fix $R_{4}>R_{1}>0.$ Note that for every sufficiently large $R>0$, we have
	\[ u_{R}^{q+1}(x) + v_{R}^{p+1}(x) \le C R^{(q+1)\tilde{\alpha}} |Rx|^{-(q+1)\tilde{\alpha}} + C R^{(p+1)\tilde{\beta}} |Rx|^{-(p+1)\tilde{\beta}} \le C, \quad \f x \in B_{R_{4}} \!\setminus\! B_{R_{1}}, \]
	where $C>0$ is independent of $R.$ Therefore one can use Theorem \ref{1.6} to complete the proof.
\end{proof}

	The proof of Theorem \ref{1.4}\,(iii) is more complicated compared to that of Theorem \ref{1.4}\,(ii). In fact, we will use Theorem \ref{1.4}\,(ii) to establish Theorem \ref{1.4}\,(iii). To begin, we first examine the notion of stability.

\begin{proposition}\label{3.24}
	Let $N \in \NN_+,\Omega \sub \RR^N$ be a open set and $ (f,g) \in \left( C^1(\Omega \times (0,+\oo)) \right)^{2} $ satisfying $0 \le (D_{v} f(\cdot,v(\cdot)), D_{u} g(\cdot,u(\cdot))) \in \left( C_{\rm{loc}}^{\gamma}(\Omega) \right)^{2} $ for some $\gamma \in (0,1).$ If $(u,v) \in \left( C^2({\Omega}) \right)^2$ is a stable positive solution of \eqref{eq1.4}, then for every smooth open set $\Omega' \Subset \Omega,$ there exists $ 0 < (\xi,\zeta) \in \left( C^{2,\gamma}(\ol{\Omega'}) \right)^{2}$ such that
	\begin{align}\label{eq3.31}
		\left\{\begin{aligned}
			-\Delta \xi &= D_v f(x,v(x))\zeta \\
			-\Delta \zeta &= D_u g(x,u(x))\xi
		\end{aligned}
		\quad \mbox{in}\ \  \Omega'.
		\right.\end{align}
\end{proposition}
\begin{proof}
	From the definition of stability, one can deduce that \eqref{eq3.31} admits a positive upper solution $(\varphi,\psi) \in \left( C^2(\ol{\Omega'}) \right)^2.$ Note that $(\min_{\Omega} \varphi, \min_{\Omega} \psi)  $ is a positive lower solution of \eqref{eq3.31}. Hence apply the upper and lower solutions method (see e.g. {\cite[Thm.\,4.1 and Rem.\,4.1]{zbMATH07819857}}), and the proof is complete.
\end{proof}

	Thereafter, we employ specific techniques from Poláčik--Quittner--Souplet \cite{zbMATH05196152}, Phan--Souplet \cite{zbMATH05998062}, and Phan \cite{zbMATH06121773} to derive the decay estimate \eqref{eq1.7}. In this manner, Theorem \ref{1.4}\,(iii) can be readily deduced from Theorem \ref{1.4}\,(ii).

\begin{lemma}[Doubling lemma, see {\cite[Lem.\,5.1 and Rem.\,5.2]{zbMATH05196152}}]\label{3.23}
	Let $\Omega \subsetneq \RR^N$ be an open set and $M: \Omega \ra (0,+\oo)$ be bounded on every compact subsets of $\Omega.$ If there exist $y \in \Omega$ and $k>0$ such that
	\[ M(y) \operatorname{dist\,}(y,\pt{\Omega}) > 2k, \]
	then there exists $x \in \Omega$ such that
	\[ M(x) \operatorname{dist\,}(x,\pt{\Omega}) > 2k, \quad M(x) \ge M(y), \]
	and
	\[ M(z) \le 2M(x), \quad \f z \in \ol{B(x,kM^{-1}(x))} \sub \Omega. \]
\end{lemma}

\begin{proposition}\label{3.25}
	Suppose $N \in \NN_+, p,q>0,pq>1,\Omega \subsetneq \RR^N$ is an open set and $C_1 \le c,d \in C^{1}(\Omega)$ such that
	\begin{align}\label{eq3.50}
		\begin{aligned}
			\|c\|_{C^{0,\gamma}(\ol{\Omega})} + \|d\|_{C^{0,\gamma}(\ol{\Omega})} \le C_2,
		\end{aligned}
	\end{align}
	where $C_1,C_2>0$ and $\gamma \in (0,1].$ If \eqref{LE} admits no bounded positive solution in $\left( C^{\oo}(\RR^N) \right)^{2} $ that is stable in every bounded open subset of $\RR^N$, then for every positive solution $(u,v) \in \left( C^{2}(\Omega) \right)^{2}$ of
	\begin{align}\label{eq3.51}
		\left\{\begin{aligned}
			-\Delta u &= cv^p \\
			-\Delta v &= du^q
		\end{aligned}
		\quad \mbox{in}\ \  \Omega,
		\right.\end{align}
	which is stable in every open set $\Omega' \Sub \Omega$, we have
	\[ \sum\limits_{i=0}^{1} \left(|D^{i} u|^{\frac{1}{\alpha+i}} + |D^{i} v|^{\frac{1}{\beta+i}} \right)
	\le  C\left(1+\operatorname{dist}^{-1}(\cdot,\partial\Omega)\right) \quad \mbox{in}\ \ \Omega,\]
	where $C=C(N,p,q,\gamma,C_1,C_2)>0$ is independent of $\Omega.$
\end{proposition}
\begin{proof}
	\emph{Step 1.} Assume the conclusion is false, then for every $k>0,$ there exist an open set $\Omega_{k} \subsetneq \RR^N,$ $ C_{1} \le c_{k} , d_{k} \in C^{1}({\Omega_{k}})$ satisfying \eqref{eq3.50} in $\Omega_{k}$ and a positive solution $(u_{k},v_{k}) \in \left( C^{2}(\Omega_{k}) \right)^{2} $ of \eqref{eq3.51} in $\Omega_{k}$, which is stable in every open set $\Omega_{k}' \Sub \Omega_{k}$, such that
	\[ \sum\limits_{i=0}^{1} \left(|D^{i} u_{k}(y_{k})|^{\frac{1}{\alpha+i}} + |D^{i} v_{k}(y_{k})|^{\frac{1}{\beta+i}} \right)
	>  2k \left( 1+ \operatorname{dist}^{-1}(y_{k},\partial\Omega_{k}) \right) , \]
	for some $y_{k} \in \Omega_{k} .$ Let us define
	\[ M_{k} = \sum\limits_{i=0}^{1} \left(|D^{i} u_{k}|^{\frac{1}{\alpha+i}} + |D^{i} v_{k}|^{\frac{1}{\beta+i}} \right) ,\]
	then we have
	\[ M_{k}(y_{k}) \operatorname{dist\,}(y_{k},\partial\Omega_{k}) > 2k \left( \operatorname{dist\,}(y_{k},\partial\Omega_{k}) + 1 \right) > 2k .\]
	By Lemma \ref{3.23}, there exists $x_{k} \in \Omega_{k}$ such that
	\[ M_{k}(x_{k}) \operatorname{dist\,}(x_{k},\pt{\Omega_{k}}) > 2k, \quad M_{k}(x_{k}) \ge M_{k}(y_{k}), \]
	and
	\[ M_{k}(z) \le 2M_{k}(x_{k}), \quad \f z \in \ol{B(x_{k},kM_{k}^{-1}(x_{k}))} \sub \Omega_{k}. \]
	Since $M_{k}(x_{k}) \ge M_{k}(y_{k}) > 2k,$ we deduce
	\begin{align}\label{eq3.52}
		\begin{aligned}
			\lambda_{k} := M_{k}^{-1} (x_{k}) \le \frac{1}{2k} \ra 0, \quad k \ra +\oo.
		\end{aligned}
	\end{align}
	Note that $x_{k} + \lambda_{k} y \in \ol{B(x_{k},k\lambda_{k})} \sub \Omega_{k} $ provided $y \in \ol{B_{k}} , $ thus we can define
	\[ \tilde{u}_{k}(y) = \lambda_{k}^{\alpha} u_{k} (x_{k} + \lambda_{k} y),\,
	\tilde{v}_{k}(y) = \lambda_{k}^{\beta} v_{k} (x_{k} + \lambda_{k} y),\,
	\tilde{c}_{k}(y) = c_{k} (x_{k} + \lambda_{k} y),\,
	\tilde{d}_{k}(y) = d_{k} (x_{k} + \lambda_{k} y), \]
	where $y \in \ol{B_{k}}.$ Then it is easily seen that $(\tilde{u}_{k}, \tilde{v}_{k}) \in \left( C^{2}(\ol{B_{k}}) \right)^{2}$ is a positive solution of
	\begin{align*}
		\left\{\begin{aligned}
			-\Delta \tilde{u}_k &= \tilde{c}_{k} {\tilde{v}}_{k}^p \\
			-\Delta \tilde{v}_{k} &= \tilde{d}_{k} \tilde{u}_k^q
		\end{aligned}
		\quad \mbox{in}\ \  \ol{B_{k}}.
		\right.
	\end{align*}
	and
	\begin{align}\label{eq3.55}
		\begin{aligned}
			|\tilde{u}_k|^{\frac{1}{\alpha}} + |\tilde{v}_{k}|^{\frac{1}{\beta}} \le 2  \quad \mbox{in}\ \ B_{k},
			\quad
			\sum\limits_{i=0}^{1} \left(|D^{i} \tilde{u}_{k}(0)|^{\frac{1}{\alpha+i}} + |D^{i} \tilde{v}_{k}(0)|^{\frac{1}{\beta+i}} \right) = 1.
		\end{aligned}
	\end{align}
	\emph{Step 2.} Fix $R>0.$ Then \eqref{eq3.52} implies that there exists $k_{0}(R)>0$ such that $\lambda_{k} \le 1$ for every $k>k_{0}(R).$ It follows from this and \eqref{eq3.50} that
	\begin{align}\label{eq3.53}
		\begin{aligned}
			|\tilde{c}_{k} (y_{1}) - \tilde{c}_{k} (y_{2})| + |\tilde{d}_{k} (y_{1}) - \tilde{d}_{k} (y_{2})| \le C_{2} \lambda_{k}^{\gamma} |y_{1}-y_{2}|^{\gamma} \le C_{2} |y_{1}-y_{2}|^{\gamma} , \quad \f y_{1},y_{2} \in \ol{B_{R}}.
		\end{aligned}
	\end{align}
	Therefore, by the Arzela-Ascoli theorem, there exist $\tilde{c},\tilde{d} \in C(\RR^N)$ such that $\tilde{c}_{k} \ra \tilde{c}$ and $\tilde{d}_{k} \ra \tilde{d}$ in $C_{\rm{loc}}(\RR^N)$ after extracting a subsequence. Moreover, using \eqref{eq3.52} and \eqref{eq3.53}, we see that
	\[ |\tilde{c} (y_{1}) - \tilde{c} (y_{2})| + |\tilde{d} (y_{1}) - \tilde{d} (y_{2})| = \lim\limits_{k \ra +\oo} \left( |\tilde{c}_{k} (y_{1}) - \tilde{c}_{k} (y_{2})| + |\tilde{d}_{k} (y_{1}) - \tilde{d}_{k} (y_{2})| \right) \le 0 ,\]
	for every $y_{1},y_{2} \in \RR^N.$
	Therefore $\tilde{c} \equiv \tilde{c}(0) \ge C_{1}>0$ and $\tilde{d} \equiv \tilde{d}(0) \ge C_{1}>0$ in $\RR^N.$
	
	Now let us fix $m>N$. Then using the elliptic $L^p$-estimates and Eberlein-Smulian theorem, we deduce that $\tilde{{u}}_{k} \rightharpoonup \tilde{u}$ and $ \tilde{v}_{k} \rightharpoonup \tilde{v} $ in $W_{\rm{loc}}^{2,m}(\RR^N)$ after extracting a subsequence. Furthermore, by the Sobolev embeddings, we see that $ \tilde{u}_{k} \ra \tilde{u}$ and $\tilde{v}_{k} \ra \tilde{v} $ in $C_{\rm{loc}}^{1}(\RR^N)$. For this reason, it is easy to check that $(\tilde{u},\tilde{v}) \in \left( W_{\rm{loc}}^{2,m}(\RR^N) \right)^{2} $ is a bounded nonnegative strong solution of
	\begin{align}\label{eq3.54}
		\left\{\begin{aligned}
			-\Delta \tilde{u} &= \tilde{c}(0) {\tilde{v}}^p \\
			-\Delta \tilde{v} &= \tilde{d}(0) \tilde{u}^q
		\end{aligned}
		\quad \mbox{in}\ \  \RR^N.
		\right.
	\end{align}
	From the strong maximum principle for strong solutions (see e.g. {\cite[Thm.\,9.6]{gilbargtrudinger1977}}) and \eqref{eq3.55}, it follows that $(\tilde{u},\tilde{v})$ is positive. Then it is immediate that $(\tilde{u},\tilde{v}) \in \left( C^{\oo}(\RR^N) \right)^{2}$ by the elliptic regularity theory.
	\\
	\emph{Step 3.} We claim that $(\tilde{u},\tilde{v})$ is also stable in every bounded open subset of $\RR^N$. Indeed, Proposition \ref{3.24} implies that there exist $0< \xi_{k},\zeta_{k} \in  C^{2}(\ol{B(x_{k},k\lambda_{k})})$ such that
	\begin{align*}
		\left\{\begin{aligned}
			-\Delta \xi_{k} &= p {c}_{k} v_{k}^{p-1} \zeta_{k} \\
			-\Delta \zeta_{k} &= q {d}_{k} u_{k}^{q-1} \xi_{k}
		\end{aligned}
		\quad \mbox{in}\ \  B(x_{k},k\lambda_{k}).
		\right.
	\end{align*}
	Let
	\[ \tilde{\xi}_{k}(y) = \lambda_{k}^{\alpha} \xi_{k}(x_{k}+\lambda_{k}y), \quad  \tilde{\zeta}_{k}(y) = \lambda_{k}^{\beta} \zeta_{k}(x_{k}+\lambda_{k}y), \quad y \in \ol{B_{k}}.\]
	Then it is clear that
	\begin{align}\label{eq3.57}
		\left\{\begin{aligned}
			-\Delta \tilde{\xi}_{k} &= p \tilde{c}_{k} \tilde{v}_{k}^{p-1} \tilde{\zeta}_{k} \\
			-\Delta \tilde{\zeta}_{k} &= q \tilde{d}_{k} \tilde{u}_{k}^{q-1} \tilde{\xi}_{k}
		\end{aligned}
		\quad \mbox{in}\ \  B_{k}.
		\right.
	\end{align}
	For every $k>R>0,$ let us define
	\[ \tilde{\xi}_{k,R} = \frac{\tilde{\xi}_{k}}{\max_{B_{R}} \tilde{\xi}_{k} + \max_{B_{R}} \tilde{\zeta}_{k}}, \quad \tilde{\zeta}_{k,R} = \frac{\tilde{\zeta}_{k}}{\max_{B_{R}} \tilde{\xi}_{k} + \max_{B_{R}} \tilde{\zeta}_{k}}. \]
	It follows that $(\tilde{\xi}_{k,R},\tilde{\zeta}_{k,R}) \in C^{2}(\ol{B_{R}})$ is a positive solution of \eqref{eq3.57} and
	\begin{align}\label{eq3.56}
		\begin{aligned}
			\max_{B_{R}} \tilde{\xi}_{k,R} + \max_{B_{R}} \tilde{\zeta}_{k,R} = 1.
		\end{aligned}
	\end{align}
	Since $\tilde{u}_{k} \ra \tilde{u}$ and $ \tilde{v}_{k} \ra \tilde{v} $ in $C_{\rm{loc}}^{1}(\RR^N)$, and $ (\tilde{u},\tilde{v}) $ is positive in $\RR^N ,$ we see that $(\tilde{{u}}_{k}, \tilde{v}_{k})$ is locally bounded away from $0 .$ Hence similar to the deduction in Step 2, using \eqref{eq3.56} instead of \eqref{eq3.55}, we can deduce that there exists a positive classical solution $(\tilde{\xi}_{R},\tilde{\zeta}_{R}) \in \left( C^{\oo}(B_{R}) \right)^{2} $ of
	\begin{align*}
		\left\{\begin{aligned}
			-\Delta \tilde{\xi}_{R} &= p \tilde{c}(0) {\tilde{v}}^{p-1} \tilde{\zeta}_{R}  \\
			-\Delta \tilde{\zeta}_{R} &= q \tilde{d}(0) \tilde{u}^{q-1} \tilde{\xi}_{R}
		\end{aligned} \quad \mbox{in}\ \  B_{R}.
		\right.
	\end{align*}
	It follows that $(\tilde{u},\tilde{v}) \in \left( C^{\oo}(\RR^N) \right)^{2}$ is a bounded positive solution of \eqref{eq3.54}, which is stable in every bounded open subset of $\RR^N$. This leads to a contradiction.
\end{proof}

	In light of Proposition \ref{3.25} and applying the same approach as in {\cite[Proof of Thm.\,1.2]{zbMATH05998062}} and {\cite[Proof of Thm.\,1.3 and Prop.\,4.1]{zbMATH06121773}}, we obtain the following desired decay estimate. The crucial observation is that for any fixed $x \in \RR^N \!\setminus\! \ol{B_{\sigma {R_{0}}}} $ and $r=\frac{(\sigma-1)|x|}{2\sigma}$, the function
	\[ (U,V) := \left( r^{\tilde{\alpha}} u(x + r\cdot), r^{\tilde{\beta}} v(x + r\cdot) \right) \in \left( C^{2}(B_{1}) \right)^{2} \]
	is a positive solution of
	\begin{align*}
		\left\{\begin{aligned}
			-\Delta U	&= \left| \cdot+ \frac{x}{r} \right|^{a} V^p\\
			-\Delta V	&= \left| \cdot+ \frac{x}{r} \right|^{b} U^q
		\end{aligned}\right. \quad \mbox{in}\ \  B_1,
	\end{align*}
	which is stable in every open set $\Omega \Sub B_{1}$.
	
\begin{proposition}\label{3.5}
	Suppose $N \in \NN_+, p,q>0,pq>1,a,b \in \RR$ and \eqref{LE} admits no bounded positive solution in $\left( C^{\oo}(\RR^N) \right)^{2}$ that is stable in every bounded open subset of $\RR^N$. Assume $R_{0} \ge 0$ and $(u,v) \in \left( C^{2}(\RR^N \!\setminus\! \ol{B_{R_{0}}}) \right)^{2}$ is a positive solution of \eqref{HLE} in $\RR^N \!\setminus\! \ol{B_{R_{0}}}$, which is stable in every open set $\Omega \Subset \RR^N \!\setminus\! \ol{B_{R_{0}}}$. Then for every $\sigma \in (1,+\oo),$ we have
	\[ \sum\limits_{i=0}^{1} \left(|x|^{\tilde{\alpha}+i}|D^i u (x)| + |x|^{\tilde{\beta}+i}|D^i v (x)|\right)  \le C,	\quad \f x \in \RR^N \!\setminus\! \ol{B_{\sigma {R_{0}}}} ,\]
	where $\ol{B_{0}} := \{0\} $ and $C=C(N,p,q,a,b,\sigma)>0$ is independent of ${R_{0}}.$	In particular, we have
	\[ u = O(|x|^{-\tilde{\alpha}}), \quad  v = O(|x|^{-\tilde{\beta}}), \quad |x| \ra +\oo,\]
	and
	\[ |Du| = O(|x|^{-\tilde{\alpha}-1}), \quad  |Dv| = O(|x|^{-\tilde{\beta}-1}), \quad |x| \ra +\oo. \]
\end{proposition}

	Finally, we establish Proposition \ref{1.7}. Once this proposition is proven, Theorem \ref{1.8} holds, as discussed in Subsection \ref{sec1.1}. Consequently, Theorem \ref{1.4}\,(iii) follows directly from \cite[Thm.\,1.1]{zbMATH07020411} (as stated in the introduction), Theorem \ref{1.8}, Proposition \ref{3.5}, and Theorem \ref{1.4}\,(ii).
	
\begin{proof}[\bf{Proof of Proposition \ref{1.7}}]
	(i) \emph{Step 1.} Firstly, we have $pq > \frac{(q+1)q}{3q-1} \ge 1.$ Furthermore, it is easy to verify that
	\[ G\left( \frac{q+1}{2} \right) = - \frac{(q+1)^2(5pq+p+q+1)[(3q-1)p-q-1]}{16(p+1)^2}<0,\]
	and
	\[G'\left( \frac{q+1}{2} \right) = \frac{-(q+1) \left[(7 q^2+2 q-1)p^2 + 2(q^2-2q-1)p - (q+1)^2\right]}{2(p+1)^2} :=\frac{-(q+1) f_q(p)}{2(p+1)^2}. \]
	Note that $3q-1>0, 7 q^2+2 q-1>0 $ and
	\[(q+1)(7q^2+2q-1) + (3q-1)(q^2-2q-1) = 2q^2(5q+1)>0,\]
	thus $p>\frac{q+1}{3q-1} > -\frac{q^2-2q-1}{7q^2+2q-1}$ and $f_{q}$ is strictly increasing on $\left( \frac{q+1}{3q-1},+\oo \right)$. It follows that
	\begin{align*}
		f_q(p) > f_q\left(\frac{q+1}{3q-1} \right)
		=\frac{4q(q-1)^2(q+1)}{(3q-1)^2}
		\ge 0.
	\end{align*}
	For this reason, we have $G'\left( \frac{q+1}{2} \right)<0.$ Since
	\[ G''(x)= 12x^2-\frac{8pq(q+1)}{p+1},\]
	we know that $G'$ has at most one zero on $\left[ \frac{q+1}{2},+\oo \right).$ But $\lim\limits_{x \ra +\oo} G(x) = \lim\limits_{x \ra +\oo} G'(x) = +\oo,$ it is immediate that $G(A)<0$ provided $A \in \left[ \frac{q+1}{2},z_{0} \right)$, and $G(A) \ge 0$ provided $A \in [z_{0},+\oo)$. \\
	\emph{Step 2.} Now we show $z_{0} > \frac{q+1}{2} + \frac{1}{\alpha}.$ It suffices to show $G\left( \frac{q+1}{2} + \frac{1}{\alpha} \right)<0$ due to the above discussion. Since $ G\left( \frac{q+1}{2} + \frac{1}{\alpha} \right) = \frac{g_q(p)}{16(p+1)^{4}},$ where
	\begin{align*}
		\begin{aligned}
			{g_q(p)} &= q^4 + 2 p^2 q^2 (35 + 28 q - 4 q^2) + 4 p q (-4 + 5 q^2 + 2 q^3) \\
			&\quad -
			4 p^3 q (-5 - 14 q + 8 q^2 + 16 q^3) +
			p^4 (1 + 8 q - 8 q^2 - 64 q^3 - 48 q^4),
		\end{aligned}
	\end{align*}
	we are reduced to prove $g_q(p)<0.$ Note that
	\begin{align*}
		\begin{aligned}
			g_q^{(4)}(p) = 24(1 + 8 q - 8 q^2 - 64 q^3 - 48 q^4) <  g_q^{(4)}\left( \frac{1}{3} \right)
			= -\frac{120}{27}
			<0
		\end{aligned}
	\end{align*}
	due to $q > \frac{1}{3},$ hence $ g_q^{(3)}$ is strictly decreasing on $\left( \frac{q+1}{3q-1}, +\oo \right).$ Moreover, we have
	\begin{align*}
		\begin{aligned}
			g_q^{(3)}(p) &= -24 q (-5 - 14 q + 8 q^2 + 16 q^3) +
			24 p (1 + 8 q - 8 q^2 - 64 q^3 - 48 q^4), \\
			g_q^{(3)}\left( \frac{q+1}{3q-1} \right) &= -\frac{24 (-1 - 4 q - q^2 + 22 q^3 + 120 q^4 + 96 q^5)}{-1 + 3 q} < 0,  \\
			g_q''(p) &= 4 q^2 (35 + 28 q - 4 q^2) - 24 p q (-5 - 14 q + 8 q^2 + 16 q^3) \\
			&\quad +
			12 p^2 (1 + 8 q - 8 q^2 - 64 q^3 - 48 q^4),	\\
			g_q''\left( \frac{q+1}{3q-1} \right) &= -\frac{4 (-3 - 38 q^2 + 92 q^3 + 157 q^4 + 540 q^5 + 468 q^6)}{(1 - 3 q)^2}
			<0,  \\
			g_q'(p) &= 4 p q^2 (35 + 28 q - 4 q^2) + 4 q (-4 + 5 q^2 + 2 q^3) \\
			&\quad -
			12 p^2 q (-5 - 14 q + 8 q^2 + 16 q^3) +
			4 p^3 (1 + 8 q - 8 q^2 - 64 q^3 - 48 q^4),	\\
			g_q'\left( \frac{q+1}{3q-1} \right) &= -\frac{4 (-1 + 9 q^2 - 34 q^3 + 81 q^4 + 84 q^5 + 199 q^6 +
				174 q^7)}{(-1 + 3 q)^3}<0.
		\end{aligned}
	\end{align*}
	In this way, we see that $g_{q}$ is strictly decreasing on $\left( \frac{q+1}{3q-1},+\oo \right),$ and so
	\[ g_{q}(p)< g_q\left( \frac{q+1}{3q-1} \right) = -\frac{(-1 + q)^4 (-1 - 12 q + 42 q^2 + 84 q^3 + 15 q^4)}{(1 - 3 q)^4} \le 0. \]
	(ii) We only prove that $z_{0} > \frac{q+1}{2} + \frac{3}{\alpha}.$ Since $ G\left( \frac{q+1}{2} + \frac{3}{\alpha} \right) = \frac{h_q(p)}{16(p+1)^{4}},$ where
	\begin{align*}
		\begin{aligned}
			h_q(p) &=  (-2 + q)^4 + p^4 (1 + 16 q + 32 q^2 - 64 q^3) +
			4 p (-8 - 56 q + 18 q^2 - 11 q^3 + 4 q^4) \\
			&\quad +
			2 p^2 (12 + 36 q + 259 q^2 + 8 q^3 + 16 q^4) -
			4 p^3 (2 + 11 q - 4 q^2 + 84 q^3 + 16 q^4)  ,
		\end{aligned}
	\end{align*}
	we are reduced to prove $h_q(p)<0.$ Note that
	\begin{align*}
		\begin{aligned}
			h_q^{(4)}(p) = 24(1 + 16 q + 32 q^2 - 64 q^3 )
			\le h_q^{(4)}(1) = -360<0,
		\end{aligned}
	\end{align*}
	due to $q \ge 1,$ hence $ h_q^{(3)}$ is strictly decreasing on $[1,+\oo).$ Furthermore, we have
	\begin{align*}
		\begin{aligned}
			h_q^{(3)}(p) &= 24 p (1 + 16 q + 32 q^2 - 64 q^3) -
			24 (2 + 11 q - 4 q^2 + 84 q^3 + 16 q^4), \\
			h_q^{(3)}(1) &= -24 (1 - 5 q - 36 q^2 + 148 q^3 + 16 q^4) < 0,  \\
			h_q''(p) &= 12 p^2 (1 + 16 q + 32 q^2 - 64 q^3) +
			4 (12 + 36 q + 259 q^2 + 8 q^3 + 16 q^4)  \\
			&\quad -
			24 p (2 + 11 q - 4 q^2 + 84 q^3 + 16 q^4),	\\
			h_q''(1) &= -4 (-3 - 18 q - 379 q^2 + 688 q^3 + 80 q^4)<0,  \\
			h_q'(p) &= 4 p^3 (1 + 16 q + 32 q^2 - 64 q^3) +
			4 (-8 - 56 q + 18 q^2 - 11 q^3 + 4 q^4)  \\
			&\quad +
			4 p (12 + 36 q + 259 q^2 + 8 q^3 + 16 q^4) -
			12 p^2 (2 + 11 q - 4 q^2 + 84 q^3 + 16 q^4),	\\			
			h_q'(1) &= -4 (1 + 37 q - 321 q^2 + 319 q^3 + 28 q^4)<0.
		\end{aligned}
	\end{align*}
	Therefore $h_{q}$ is strictly decreasing on $[1,+\oo)$, which implies
	\[ h_{q}(p) \le h_q(1) = -(-1 + q) (1 - 211 q + 451 q^2 + 15 q^3) \le 0. \]
	Recall that $pq \neq 1,$ thus the conclusion follows easily.
\end{proof}

\subsection{Proof of Theorem \ref{1.9}}

	The proof is partially inspired by \cite{zbMATH00883497} and \cite{zbMATH05563881}. However, we deal with a more general case in this work. To begin with, following an approach similar to the proof of Theorem \ref{1.6}, we establish the following result.

\begin{theorem}\label{3.2}
	Let $N \ge 3$, $p,q>0$, $pq>1$ and $a,b>-2$. Assume $(p,q,a,b)$ is subcritical and $(u,v)$ is a nonnegative solution of \eqref{HLE}. Additionally, suppose there exist $R_4>R_3>R_2>R_1>0$, $s,t \in [0,1)$, $C_0>0$, and a sequence $\{r_i\}$ in $ (0,+\infty)$ with $\lim\limits_{i \to +\infty} r_i = +\infty$, such that for every $i \in \NN_+$, one can find $\tilde{R}_i \in (R_2,R_3)$ satisfying
	\begin{align}\label{eq3.73}
		\begin{aligned}
			\int_{S_{\tilde{R}_i}} u_{r_{i}}^{q+1} &\le C_0 F_{r_{i}}(R_4)^s + C_0, \\
			\int_{S_{\tilde{R}_i}} v_{r_{i}}^{p+1} &\le C_0 F_{r_{i}}(R_4)^t + C_0,
		\end{aligned}
	\end{align}
	\begin{align}\label{eq3.74}
		\begin{aligned}
			\|Du_{r_{i}}\|_{L^{1+\frac{1}{p}}(S_{\tilde{R}_i})} &\le C_0 \|Du_{r_{i}}\|_{L^{1+\frac{1}{p}}(B_{R_3} \setminus B_{R_2})}, \\
			\|Dv_{r_{i}}\|_{L^{1+\frac{1}{q}}(S_{\tilde{R}_i})} &\le C_0 \|Dv_{r_{i}}\|_{L^{1+\frac{1}{q}}(B_{R_3} \setminus B_{R_2})},
		\end{aligned}
	\end{align}
	and
	\begin{align}\label{eq3.75}
		\begin{aligned}
			F_1(R_4r_i) \le C_0 F_1(R_1r_i), \quad \f i \in \mathbb{N}_+,
		\end{aligned}
	\end{align}
	where $F_{r_{i}}$ is defined by
	\[F_{r_{i}}: (0,+\infty) \to \mathbb{R}, \quad r \mapsto \int_{B_r} |x|^b u_{r_{i}}^{q+1} + |x|^a v_{r_{i}}^{p+1}.\]
	Then $u \equiv v \equiv 0$ in $\mathbb{R}^N$.
\end{theorem}
\begin{proof}
	\emph{Step 1. Estimate of $F_{1}.$} Firstly, for every $i \in \NN_+$, by \eqref{eq3.74}, Lemma \ref{2.1}\,(ii) and Lemma \ref{2.2}\,(ii), we must have
	\begin{align}\label{eq3.37}
		\begin{aligned}
			\|Du_{r_{i}}\|_{L^{1+\frac{1}{p}}(S_{\tilde{R}_{i}}) } &\le
			C_0\|D u_{r_{i}}\|_{L^{1+\frac{1}{p}}(B_{R_{3}} \setminus B_{R_{2}})} \\
			&\le
			C \left(
			\|\Delta u_{r_{i}}\|_{L^{1+\frac{1}{p}} (B_{R_{4}} \setminus B_{1})}
			+\|u_{r_{i}}\|_{L^1(B_{R_{4}} \setminus B_{R_{1}})}
			\right)  \\
			&\le C \left( \| |x|^a v_{r_{i}}^{p} \|_{L^{1+\frac{1}{p}} (B_{R_{4}} \setminus B_{R_{1}})} + 1 \right) \\
			&\le C \left( F_{r_{i}}(R_{4})^{\frac{p}{p+1}} + 1 \right) ,
		\end{aligned}
	\end{align}
	where $C=C(N,p,q,a,b,R_{1},R_{2},R_{3},R_{4},C_0)>0.$ By \eqref{eq3.73} and \eqref{eq3.37}, we conclude that
	\begin{align}\label{eq3.43}
		\begin{aligned}
			\int_{S_{\tilde{R}_{i}}} |Du_{r_{i}}|v_{r_{i}} \le  \|Du_{r_{i}}\|_{L^{1+\frac{1}{p}}(S_{\tilde{R}_{i}})} \|v_{r_{i}}\|_{L^{p+1}(S_{\tilde{R}_{i}})} \le  C \left( F_{r_{i}}(R_{4})^{\frac{p+t}{p+1}} +1  \right)
		\end{aligned}
	\end{align}
	Likewise, we have
	\[ \|Dv_{r_{i}}\|_{L^{1+\frac{1}{q}}(S_{\tilde{R}_{i}})} \le C \left( F_{R}(R_{4})^{\frac{q}{q+1}} + 1 \right),\]
	and
	\begin{align}\label{eq3.44}
		\begin{aligned}
			\int_{S_{\tilde{R}_{i}}} u_{r_{i}}|Dv_{r_{i}}| \le C \left(  F_{r_{i}}(R_{4})^{\frac{q+s}{q+1}} +1 \right).
		\end{aligned}
	\end{align}
	Now using \eqref{eq2.7}, \eqref{eq3.43} and \eqref{eq3.44}, we see that
	\begin{align*}
		\begin{aligned}
			F_{r_{i}}(R_{1})&= \int_{B_{R_{1}}} |x|^b u_{r_{i}}^{q+1} + |x|^a v_{r_{i}}^{p+1} \\
			&\le \int_{B_{\tilde{R}_{i}}}|x|^b u_{r_{i}}^{q+1} + |x|^a v_{r_{i}}^{p+1} \\
			&\le C \int_{S_{\tilde{R}_{i}}} u_{r_{i}}^{q+1} + v_{r_{i}}^{p+1}  + |Du_{r_{i}}|v_{r_{i}} + u_{r_{i}}|Dv_{r_{i}}| \\
			& \le C \left( F_{r_{i}}(R_{4})^s + F_{r_{i}}(R_{4})^t+ F_{r_{i}}(R_{4})^{\frac{p+t}{p+1}} + F_{r_{i}}(R_{4})^{\frac{q+s}{q+1}} +1 \right) \\
			&\le C \left( F_{r_{i}}(R_{4})^{A}+1 \right) ,
		\end{aligned}
	\end{align*}
	where $A:=\max\left\{ s,t, \frac{p+t}{p+1}, \frac{q+s}{q+1}\right\} \in [0,1).$ From this and \eqref{eq3.69}, it follows that
	\[ F_1(R_{1}r_{i}) r_{i}^{\tilde{\alpha} + \tilde{\beta} -N+2} \le C \left( F_1(R_{4}r_{i})^A r_{i}^{A(\tilde{\alpha} + \tilde{\beta} -N+2)} + 1 \right), \quad \f i \in \NN_+, \]
	which implies
	\begin{align}\label{eq3.41}
		\begin{aligned}
			F_1(R_{1}r_{i})  \le C
			\left( F_1(R_{4}r_{i})^A r_{i}^{-(1-A)(\tilde{\alpha} + \tilde{\beta} -N+2)} + r_{i}^{-(\tilde{\alpha} + \tilde{\beta} -N+2)} \right), \quad \f i \in \NN_+.
		\end{aligned}
	\end{align}
	\emph{Step 2.} From \eqref{eq3.41} and \eqref{eq3.75}, we conclude that
	\begin{align}\label{eq3.45}
		\begin{aligned}
			F_1(R_{1}r_i) \le C \left( F_1(R_{1}r_i)^A r_i^{-(1-A)(\tilde{\alpha} + \tilde{\beta} -N+2)} +r_i^{-(\tilde{\alpha} + \tilde{\beta} -N+2)} \right), \quad \f i \in \NN_+.
		\end{aligned}
	\end{align}
	We thus observe that the sequence $\{F_1(R_{1}r_i)\}$ must be bounded. Otherwise, there exists a subsequence $\{F_1(R_{1}r_{i_j})\}$ such that $ F_1(R_{1}r_{i_j}) \ra +\oo$ as $j \ra +\oo$. From \eqref{eq3.45}, we have
	\begin{align}\label{eq3.46}
		\begin{aligned}
			F_1(R_{1}r_{i_j})^{1-A} \le C \left( r_{i_j}^{-(1-A)(\tilde{\alpha} + \tilde{\beta} -N+2)} + F_1(R_{1}r_{i_j})^{-A} r_{i_j}^{-(\tilde{\alpha} + \tilde{\beta} -N+2)} \right),
		\end{aligned}
	\end{align}
	which implies the right-hand term of \eqref{eq3.46} tends to $0$ as $i \ra +\oo$ due to \eqref{eq2.1}. It is evident that a contradiction arises. Therefore $\{F_1(R_{1}r_i)\} $ is bounded, and so $F_1(R_{1}r_i) \ra 0$ as $i \ra +\oo$ due to \eqref{eq3.45}. This means $u \equiv v \equiv 0$ in $\RR^N$.
\end{proof}

\begin{remark}\label{3.26}
	We observe that if $(u, v)$ is a nonnegative solution of \eqref{HLE} that is either polynomially bounded or satisfies $|x|^{a}u^{q+1} + |x|^{b}v^{p+1} \in L^{1}(\RR^N)$, then there exists a sequence $\{r_i\} $ in $ (0,+\infty)$ with $\lim\limits_{i \to +\infty} r_i = +\infty$, such that \eqref{eq3.75} holds. We shall prove the case where $(u, v)$ is polynomially bounded, as the other case is straightforward. Suppose, for the sake of contradiction, that the assertion is false. Then, for every $M > 1$, there exists $r_{0} > 1$ such that $F_1(R_{4}r) > MF_1(R_{1}r)$ for every $r \ge \frac{r_{0}}{R_{1}}$. Let $r_{i+1} = \frac{R_{4}}{R_{1}}r_{i} := R_{4,1} r_{i} $ for every $i \in \NN$ and fix
	\[ M > R_{4,1}^{a + (p+1)L + N} + R_{4,1}^{b + (q+1)L + N} + 1. \]
	Then for every $i \in \NN_+$, we have
	\begin{align}\label{eq3.14}
		\begin{aligned}
			F_1(r_{i+1}) = F_{1}(R_{4,1} r_{i}) > MF_{1}(r_{i}) > \cdots > M^{i}F_1(r_1) > M^{i+1} F_1(r_0) \ge 0.
		\end{aligned}
	\end{align}
	Since $(u,v)$ is polynomially bounded, there exist $K,L>0$ such that
	\[|u(x)|+|v(x)| \le K|x|^L, \quad \f |x| \ge 1.\]
	In this way, for every $i \in \NN_+$, we deduce that
	\[\begin{aligned}
		M^{i}F_1(r_1) &< F_1(r_{i+1}) \\
		&= \int_{B_{r_{i+1}}} |x|^b u^{q+1} +  |x|^a v^{p+1} \\
		&\le C \int_{B_{r_{i+1}} \!\setminus B_1} \left( |x|^b |x|^{(q+1)L} +  |x|^a|x|^{(p+1)L} \right) \mathrm{d}x +C \\
		&= C \int_{1}^{r_{i+1}} \left(r^{b+(q+1)L+N-1} + r^{a+(p+1)L+N-1}\right) \mathrm{d}r +C \\
		&\le C \left[ r_{1}^{b+(q+1)L+N} \left( R_{4,1}^{b+(q+1)L+N} \right)^i + r_{1}^{a+(p+1)L+N} \left( R_{4,1}^{a+(p+1)L+N} \right)^i \right] +C\\
		&\le C \left(r_{1}^{b+(q+1)L+N} + r_{1}^{a+(p+1)L+N}\right) (M-1)^i +C,
	\end{aligned}\]
	where $C>0$ is independent of $i.$ Now dividing both sides by $M^{i}$ and taking the limit as $i \ra +\oo$, we obtain $F_{1}(r_{1})=0$, which contradict to \eqref{eq3.14}.
\end{remark}

	We will next apply Theorem \ref{3.2} to prove Theorem \ref{1.9}. To facilitate the application of Theorem \ref{3.2}, we introduce the following lemma.

\begin{lemma}\label{3.20}
	Suppose $N\ge 3,p,q>0,a,b \in \RR$ and $(u,v)$ is a nonnegative solution of \eqref{HLE}. Assume there exist $R_{4}>R_{3}>\tilde{R}>R_{2}>0, q_1>0,s_1 \in [0,1)$ and $C_{1}>0$ such that \\
	(i) One of the following holds:
	\begin{itemize}
		\item $q \le q_1-1 ;$
		\item $q>q_1-1$, $p \le \frac{2}{N-3}$ and $q < q_1- \big(1+\frac{1}{p}\big)s_1 +\frac{1}{p} ;$
		\item $q>q_1-1$, $p > \frac{2}{N-3}$ and $\big( q+1-\frac{N-1}{\alpha}\big)s_1 + \frac{N-1}{\alpha}<q_1.$
	\end{itemize}
	(ii) The following inequalities are satisfied:
	\[ \int_{S_{\tilde{R}}}  u^{q_1} \le C_{1} \left( \int_{B_{R_{4}}} |x|^b u^{q+1} + |x|^a v^{p+1} \right)^{s_{1}} + C_{1} , \ \  \|D^2 u(\tilde{R}))\|_{1+\frac{1}{p}} \le C_{1}\|D^2 u\|_{L^{1+\frac{1}{p}}(B_{R_{3}} \setminus B_{R_{2}})} .\]
	Then there exist $s=s(N,p,q,q_1,s_1) \in [0,1)$ and $C_{0}=C_{0}(N,p,q,a,b,R_{2},R_{3},R_{4},C_{1})>0$ such that
	\[\int_{S_{\tilde{R}}}  u^{q+1} \le C_{0} \left( \int_{B_{R_{4}}} |x|^b u^{q+1} + |x|^a v^{p+1} \right)^{s} + C_{0} .\]
\end{lemma}

	This lemma's proof requires two supplementary inequalities. We first recall the Sobolev inequality on manifolds.

\begin{lemma}[Sobolev's inequality, see {\cite[Lem.\,2.4]{zbMATH00883497}}]\label{2.4}
	Suppose $M$ is a compact connected $C^{\oo}$ Riemannian manifold with or without boundary and $w \in W^{j,k}(M),$ where $j \in \NN_+$ and $ k \in [1,+\oo].$ Assume $k_{j}^{*} \in [1,+\oo]$ satisfies
	\[
	\left\{
	\begin{aligned}
		& k_j^*= \dfrac{(N-1)k}{N-1-jk}, & \quad \mbox{if}\,\,\, jk<\dim M, \\
		& k_j^* \in [1,+\oo) , & \quad \mbox{if}\,\,\, jk=\dim M, \\
		& k_j^*= +\oo, & \quad \mbox{if}\,\,\, jk>\dim M.
	\end{aligned}
	\right.
	\]
	Then we have
	\[ \|w\|_{L^{k_{j}^{*}}(M)} \le C\left(\|D_{M}^{\,j} w\|_{L^{k}(M)}  +\|w\|_{L^{1}(M)} \right),\]
	where $C=C(M,j,k,k_{j}^{*})>0 .$
	
\end{lemma}

	By combining the classical $L^{p}$-estimate with Lemma \ref{2.1}\,(i), we obtain the second inequality needed.
	
\begin{lemma}[A variant of the $L^p$-estimate]\label{2.20}
	Let $N \ge 3, \Omega \sub \RR^N$ be an open set and $K \sub \Omega$ be a compact set. If $ k \in (1, +\infty)$ and $w \in$ $W^{2, k}(\Omega)$, then
	\[
	\|D^2 w\|_{L^k(K)} \le
	C\left(\|\Delta w\|_{L^k(\Omega)} + \|w\|_{L^1(\Omega)}\right),
	\]
	where $C=C(N,k,\Omega,K)>0.$
\end{lemma}

\begin{proof}[\bf{Proof of Lemma \ref{3.20}.}]
	Let us define
	\[F: (0,+\oo) \ra \RR, \quad r \mapsto \int_{B_r} |x|^b u^{q+1} + |x|^a v^{p+1}.\]
	\emph{Case 1.} If $q \le q_1-1,$ then by H\"{o}lder's inequality, we obtain
	\[ \int_{S_{\tilde{R}}} u^{q+1} \le C_{0} \left( \int_{S_{\tilde{R}}} u^{q_1} \right)^{\frac{q+1}{q_{1}}} \le C_{0} F(R_{4})^{\frac{(q+1)s_1}{q_{1}}} + C_{0}.\]
	\emph{Case 2.} If $q > q_1-1,$ then let $\lambda = \big(1+\frac{1}{p}\big)_2^*.$ We first use Lemmas \ref{2.4}, \ref{2.2}\,(ii) and \ref{2.20} to deduce
	\[\begin{aligned}
		\|u(\tilde{R})\|_{\lambda} &\le C_{0} \left(\|D_{S_1}^2 u(\tilde{R})\|_{1+\frac{1}{p}} + \|u(\tilde{R})\|_1 \right) \\
		&\le C_{0} \left(\tilde{R}^2 \|D^2 u(\tilde{R})\|_{1+\frac{1}{p}} +1 \right) \\
		&\le C_{0} \left( \|D^2 u\|_{L^{1+\frac{1}{p}}(B_{R_{3}} \setminus B_{R_{2}})} +1 \right) \\
		&\le C_{0} \left(\|\Delta u\|_{L^{1+\frac{1}{p}}(B_{R_{4}} \setminus B_{R_{1}})} + \|u\|_{L^{1}(B_{R_{4}} \setminus B_{R_{1}})} + 1\right) \\
		&\le C_{0} \left(F(R_{4})^{\frac{p}{p+1}} + 1\right),
	\end{aligned}\]
	where $C_{0}=C_{0}(N,p,q,a,b,R_{2},R_{3},R_{4},C_{1},\lambda)>0.$ \\
	\emph{Subcase 1.} If $p \le \frac{2}{N-3}, $ i.e. $2\big(1+\frac{1}{p}\big) \ge N-1,$ then fixing $\lambda>q+1$ large enough and using the interpolation inequality, we obtain
	\[\|u(\tilde{R})\|_{q+1} \le \|u(\tilde{R})\|_{q_1}^{1-h} \|u(\tilde{R})\|_{\lambda}^{h} \le C_{0}\left(F(R_{4})^{(1-h) \frac{s_1}{q_1} + h \frac{p}{p+1}} +1 \right),\]
	where $C_{0}=C_{0}(N,p,q,a,b,R_{2},R_{3},R_{4},C_{1})>0$ and
	\[h=h(\lambda)=\frac{\frac{1}{q_1} - \frac{1}{q+1} }{ \frac{1}{q_1} - \frac{1}{\lambda}} \in (0,1).\] It remains to show
	\[ \left[(1-h) \frac{s_1}{q_1} + h \frac{p}{p+1}\right] (q+1) <1. \]
	Since $\lambda$ can be arbitrarily large, we only need to show
	\[ \lim\limits_{\lambda \ra \oo} \left[(1-h) \frac{s_1}{q_1} + h \frac{p}{p+1}\right] (q+1) <1.\]
	A straightforward calculation implies that this is equivalent to \[q < q_1-\left(1+\frac{1}{p}\right)s_1 +\frac{1}{p},\]
	which is exactly our condition.\\
	\emph{Subcase 2.} If $p > \frac{2}{N-3}, $ i.e. $2(1+\frac{1}{p}) < N-1,$ then $C_{0}=C_{0}(N,p,q,a,b,R_{2},R_{3},R_{4},C_{1})$ by the definition of $\lambda$. Since
	\[\left(q+1-\frac{N-1}{\alpha}\right) s_1+ \frac{N-1}{\alpha}<q_1,\]
	it follows that \[\frac{\frac{q_1}{q+1} - s_1}{1-s_1} > \frac{N-1}{\alpha + \beta +2}.\]
	But $q+1>q_1,$ hence $\alpha + \beta +2 >N-1.$ Now we must have
	\[\frac{1}{1+\frac{1}{p}} - \frac{1}{q+1} = \frac{pq-1}{(p+1)(q+1)} = \frac{2}{(p+1)\beta} =\frac{2}{\alpha+\beta+2} < \frac{2}{N-1}.\]
	This means $q+1< \big(1+\frac{1}{p}\big)_2^*=\lambda,$ thus the interpolation inequality yields that
	\[\|u(\tilde{R})\|_{q+1} \le \|u(\tilde{R})\|_{q_1}^{1-h} \|u(\tilde{R})\|_{\lambda}^{h} \le C_{0}\left(F(R_{4})^{(1-h) \frac{s_1}{q_1} + h \frac{p}{p+1}} +1 \right),\]
	where \[h=\frac{\frac{1}{q_1} - \frac{1}{q+1} }{ \frac{1}{q_1} - \frac{1}{\lambda}} \in (0,1).\] It remains to show
	\[ \left[(1-h) \frac{s_1}{q_1} + h \frac{p}{p+1}\right] (q+1) <1. \]
	In fact, one can easily check that this is equivalent to
	\[\left(q+1-\frac{N-1}{\alpha}\right) s_1+ \frac{N-1}{\alpha}<q_1,\]
	which precisely matches our condition. The proof is finished.
\end{proof}

	Then we will prove Theorem \ref{1.9} by analyzing the range of $s$. In one case, we shall utilize the following generalization of {\cite[Thm.\,1.1]{zbMATH06121773}}.

\begin{theorem}\label{3.27}
	Let $N \ge 3,p \ge q >0, pq>1$ and $a,b >-2.$ Assume that $(p,q,a,b)$ is subcritical, and that \eqref{eq1.3} holds provided $N\ge 4$. Additionally, suppose $(u,v)$ is a nonnegative solution of \eqref{HLE} that either is polynomially bounded or satisfies $|x|^{b}u^{q+1} + |x|^{a}v^{p+1} \in L^{1}(\RR^N)$. If $\alpha>N-3$, then $u \equiv v \equiv 0$ in $\mathbb{R}^N$.
\end{theorem}
\begin{proof}
	For every $R>0$ and $\gamma \in (0,1)$, by Lemma \ref{2.30} and Lemma \ref{2.2}\,(iii),  one can find $\tilde{R} \in (2,3)$ satisfying
	\begin{align}\label{eq3.85}
		\begin{aligned}
			\|D u_{R}(\tilde{R}))\|_{1+\frac{1}{p}} \le C \|D u_{R}\|_{L^{1+\frac{1}{p}}(B_{3} \setminus B_{2})} , & \quad
			\|D v_{R}(\tilde{R}))\|_{1+\frac{1}{q}} \le C \|D v_{R}\|_{L^{1+\frac{1}{q}}(B_{3} \setminus B_{2})}, \\
			\|D^2 u_{R}(\tilde{R}))\|_{1+\frac{1}{p}} \le C \|D^2 u_{R}\|_{L^{1+\frac{1}{p}}(B_{3} \setminus B_{2})}, & \quad
			\|D^2 v_{R}(\tilde{R}))\|_{1+\frac{1}{q}} \le C \|D^2 v_{R}\|_{L^{1+\frac{1}{q}}(B_{3} \setminus B_{2})},
		\end{aligned}
	\end{align}
	and
	\begin{align}\label{eq3.86}
		\begin{aligned}
			\| Du_{R}^{\frac{\gamma}{2}} (\tilde{R}) \|_{2} \le C \| Du_{R}^{\frac{\gamma}{2}} \|_{L^{2}(B_{3} \setminus B_{2})} \le C, \quad
			\| Dv_{R}^{\frac{\gamma}{2}} (\tilde{R}) \|_{2} \le C \| Dv_{R}^{\frac{\gamma}{2}} \|_{L^{2}(B_{3} \setminus B_{2})} \le C,
		\end{aligned}
	\end{align}
	where $C>0$ is independent of $R$. Then using Lemma \ref{2.4}, \eqref{eq3.86} and Lemma \ref{2.2}\,(ii), we have
	\begin{align}\label{eq3.87}
		\begin{aligned}
			\| u_{R}^{\frac{\gamma}{2}} (\tilde{R}) \|_{2_{1}^{*}} \le C  \left( \| Du_{R}^{\frac{\gamma}{2}} (\tilde{R}) \|_{2} +  \| u_{R}^{\frac{\gamma}{2}} (\tilde{R}) \|_{1}   \right)
			\le C,
		\end{aligned}
	\end{align}
	and
	\begin{align}\label{eq3.88}
		\begin{aligned}
			\| v_{R}^{\frac{\gamma}{2}} (\tilde{R}) \|_{2_{1}^{*}} \le C  \left( \| Dv_{R}^{\frac{\gamma}{2}} (\tilde{R}) \|_{2} +  \| v_{R}^{\frac{\gamma}{2}} (\tilde{R}) \|_{1}   \right)
			\le C.
		\end{aligned}
	\end{align}
	\emph{Case 1.} If \( N = 3 \), then by the definition of \( 2_{1}^{*} \), along with \eqref{eq3.85}, \eqref{eq3.87}, \eqref{eq3.88}, and Remark \ref{3.26}, we can apply Theorem \ref{3.2} to complete the proof. \\
	\emph{Case 2.} If \( N \geq 4 \), we select \( q_{1} = \frac{\gamma(N-1)}{N-3} \), ensuring that the conditions of Lemma \ref{3.20} are satisfied, provided that \( \gamma \) is sufficiently close to 1, as guaranteed by $\alpha>N-3$, \eqref{eq3.85} and \eqref{eq3.87}. Consequently, we can use Lemma \ref{3.20} to deduce that there exist $s \in [0,1)$ and $C>0$ such that
	\begin{align}\label{eq3.89}
		\begin{aligned}
			\int_{S_{\tilde{R}}}  u_{R}^{q+1} \le C \left(\int_{B_{R_{4}}} |x|^b u_{R}^{q+1}+ |x|^a v_{R}^{p+1}\right)^{s} + C,\quad \f R>0.
		\end{aligned}
	\end{align}
	At this point, one can also apply Theorem \ref{3.2} to complete the proof, by virtue of \eqref{eq3.85}, \eqref{eq3.89}, Lemma \ref{2.3}, and Remark \ref{3.26}.
\end{proof}
	
	In the end, we are prepared to prove Theorem \ref{1.9}.

\begin{proof}[\bf{Proof of Theorem \ref{1.9}}]
	In what follows, we fix $R_{3},R_{2}>0$ such that $R_{4}>R_{3}>R_{2}>R_{1}$ and use $C>0$ to denote a general constant that is independent of $R$. \\
	\emph{Case 1.} If $p \ge s,$ observe that \[N-\alpha-2 = N-p\beta \le N-s\beta <1.\]
	Hence one can apply Theorem \ref{3.27} to finish the proof. \\
	\emph{Case 2.} If $s > p$, then by combining Lemma \ref{2.20}, Lemma \ref{2.2}\,(ii) and \eqref{eq1.8}, we have
	\[\begin{aligned}
		\|D^2 u_R\|_{L^{\frac{s}{p}} (B_{R_{3}} \setminus B_{R_{2}})} &\le C\left( \|\Delta u_R\|_{L^{\frac{s}{p}} (B_{R_{4}} \setminus B_{R_{1}})} + \|u_R\|_{L^{1} (B_{R_{4}} \setminus B_{R_{1}})}\right) \\
		&\le C \left( \int_{B_{R_{4}} \setminus B_{R_{1}} } v_R^s \right)^{\frac{p}{s}} + C \\
		&\le C,
	\end{aligned}\]
	where $R>0$ is sufficiently large. By Lemma \ref{2.30}, there exists $\tilde{R} \in (R_{2},R_{3})$ such that
	\[\|Du_R(\tilde{R}))\|_{1+\frac{1}{p}} \le C\|Du_R\|_{L^{1+\frac{1}{p}}(B_{R_{3}} \setminus B_{R_{2}})}, \quad
	\|Dv_R(\tilde{R}))\|_{1+\frac{1}{q} } \le C\|Dv_R\|_{L^{1+\frac{1}{q}}(B_{R_{3}} \setminus B_{R_{2}})},\]
	and
	\begin{align}\label{eq3.39}
		\begin{aligned}
			\|D^2 u_R(\tilde{R}))\|_{1+\frac{1}{p}} \le C\|D^2 u_R\|_{L^{1+\frac{1}{p}}(B_{R_{3}} \setminus B_{R_{2}})},
			\quad
			\|D^2 u_R(\tilde{R})\|_{\frac{s}{p}} \le C\|D^2 u_R\|_{L^{\frac{s}{p}} (B_{R_{3}} \setminus B_{R_{2}})} \le C.
		\end{aligned}
	\end{align}
	Thus, by Lemma \ref{2.4}, Lemma \ref{2.2}\,(ii) and \eqref{eq3.39}, we obtain
	\begin{align}\label{eq3.40}
		\begin{aligned}
			\|u_R(\tilde{R})\|_{(\frac{s}{p})_2^*} &\le C\left(\|D_{S_1}^2 u_R (\tilde{R})\|_{\frac{s}{p}} + \|u_R(\tilde{R})\|_1 \right)
			\le C\left(\tilde{R}^2\|D^2 u_R(\tilde{R})\|_{{\frac{s}{p}} } + 1\right)
			\le C.	
		\end{aligned}
	\end{align}
	It remains to demonstrate that the conditions of Lemma \ref{3.20} are met. In this case, Lemma \ref{3.20}, Lemma \ref{2.3}, Theorem \ref{3.2}, and Remark \ref{3.26} can be applied to deduce the desired conclusion. \\
	\emph{Subcase 1.} If $ N-1 \le \frac{2s}{p},$ then by the definition of $\big(\frac{s}{p}\big)_2^*$ and \eqref{eq3.40}, it is easy to see that there exists $q_1>0$ such that $q \le q_1-1$ and
	\[ \int_{S_{\tilde{R}}}  u_R^{q_1} \le C, \quad \f R \gg 1. \]
	For this reason, the conditions of Lemma \ref{3.20} are satisfied. \\
	\emph{Subcase 2.} If $N-1 > \frac{2s}{p},$ then we can assume $q>\big(\frac{s}{p}\big)_2^* - 1$.
	In this way, we have
	\[p \ge q > \left(\frac{s}{p}\right)_2^* -1=\frac{\frac{(N-1)s}{p}}{N-1-\frac{2s}{p}} -1 > \frac{N-1}{N-1-2}-1 =\frac{2}{N-3}. \] Furthermore, it is easy to check that
	\[ N-s\beta<1 \Longleftrightarrow \frac{N-1}{\alpha} < \left(\frac{s}{p}\right)_2^*, \]
	since $p\beta = \alpha+ 2.$ Consequently, the conditions of Lemma \ref{3.20} are satisfied in this subcase as well.
\end{proof}
\begin{remark}
	From the preceding proof, it is evident that the restriction \eqref{eq1.3} in Theorem \ref{1.9} can be replaced by the requirement that $(u,v)$ satisfies \eqref{eq2.2}. Moreover, the condition that $(u,v)$ is polynomially bounded or satisfies $|x|^{a} u^{q+1} + |x|^{b} v^{p+1} \in L^1(\mathbb{R}^N)$ can be replaced by the following: there exists a sequence $\{r_i\} $ in $ (0, +\infty)$ such that $\lim\limits_{i \to +\infty} r_i = +\infty$ and \eqref{eq3.75} holds.
\end{remark}

\section{The supercritical case: Proof of Theorems \ref{1.10}}\label{sec.4}

	In establishing the Liouville-type theorem for the supercritical case, we employ the following identity in place of inequality \eqref{eq2.7}.

\begin{lemma}[Rellich-Pokhozhaev type identity, see {\cite[Lem.\,2.6]{zbMATH06121773}}]\label{2.12}
	Assume $N \ge 3,
	p,q>0,$  $a,b>-2,$ $ 0 \in \Omega \sub \RR^N$ is an open set
	and $(u, v) \in \left( C^2(\Omega \!\setminus\! \{0\} ) \cap C(\Omega) \right)^2$ is a nonnegative solution of \eqref{HLE} in $\Omega \!\setminus\! \{0\}$. Then for every $R>0$ with $B_R \Subset \Omega$ and $c_1,c_2 \in \mathbb{R}$ with $c_1+c_2=N-2,$ there holds
	\[
	\begin{aligned}
		& \left(\dfrac{N+b}{q+1}-c_1\right) \int_{B_R}|x|^b u^{q+1}
		+\left(\dfrac{N+a}{p+1}-c_2\right) \int_{B_R}|x|^a v^{p+1} \\
		= &\, \dfrac{R^{b+1}}{q+1} \int_{S_R} u^{q+1}
		+\dfrac{R^{a+1}}{p+1} \int_{S_R} v^{p+1}
		+\int_{S_R}\left(c_1\dfrac{\partial v}{\partial \nu}u + c_2\dfrac{\partial u}{\partial \nu} v\right)
		+R\int_{S_R} \left( 2 \dfrac{\partial u}{\partial \nu} \dfrac{\partial v}{\partial \nu} - DuDv \right)
	\end{aligned}
	\]
\end{lemma}

	 First of all, we utilize some insights from Farina \cite{zbMATH05163856} and Dupaigne--Ghergu--Hajlaoui \cite{dupaignegherguhajlaoui2025} to prove the following result. To tackle the additional challenge introduced by the weight $|x|$ in system \eqref{HLE}, we present a decay estimate \eqref{eq4.1}, which is guaranteed by Proposition \ref{3.5}.
	
\begin{theorem}\label{4.4}
	Suppose $N \ge 3,p\ge q>0, pq>1$ and $a,b >-2$. Assume that $(u,v)$ is a nonnegative solution of \eqref{HLE} that satisfies \eqref{eq2.2}, $u \in L^{\frac{N}{\tilde{\alpha}}}(\RR^N)$ and
	\begin{align}\label{eq4.1}
		\begin{aligned}
			u=O(|x|^{-\tilde{\alpha}}),\quad |x| \ra +\oo.
		\end{aligned}
	\end{align}
	(i) If $\tilde{\alpha}<N$, then
	\begin{align}\label{eq4.2}
		\begin{aligned}
			u=o(|x|^{-\tilde{\alpha}}),\quad |x| \ra +\oo.
		\end{aligned}
	\end{align}
	(ii) If $\tilde{\alpha}<N$, then for every sufficiently small $\varepsilon>0,$ we have
	\begin{align}\label{eq4.3}
		\begin{aligned}
			u = O(|x|^{-N+2+\varepsilon}),\quad  v=o(|x|^{-\tilde{\beta}}),\quad |Dv|=o(|x|^{-\tilde{\beta}-1}),\quad |x| \ra +\oo.
		\end{aligned}
	\end{align}
	(iii) If $\tilde{\alpha}<N-2,$ then for every sufficiently small $\varepsilon>0,$ we have
	\begin{align}\label{eq4.4}
		\begin{aligned}
			|Du| = O(|x|^{-N+1+\varepsilon}) ,\quad |x| \ra +\oo.
		\end{aligned}
	\end{align}
	(iv) If $(p,q,a,b)$ is supercritical, then $u \equiv v \equiv 0$ in $\RR^N.$
\end{theorem}
\begin{proof}
	In the following proofs of (i)--(iii), without loss of generality, we may assume that $(u,v)$ is positive in $\RR^N$, as per Remark \ref{1.5}. \\
	(i) For every $\varepsilon>0$, we can choose $R_{0}>0$ and $C_{0}>0$ such that
	\begin{align}\label{eq4.5}
		\begin{aligned}
			\|u\|_{L^{\frac{N}{\tilde{\alpha}}}(\RR^N \setminus B_{R_{0}})} \le \varepsilon, \quad
			u(x) \le C_{0} |x|^{-\tilde{\alpha}}, \quad \f x \in \RR^N \!\setminus\! B_{R_{0}}.
		\end{aligned}
	\end{align}
	Fix $|y|> 4R_{0}$. Then it is easy to check that $B\left( y,\frac{|y|}{2} \right) \Sub  \RR^N \!\setminus\! \ol{B_{R_{0}}} $. Recall that
	\[ -\Delta u - lu=0 \quad \mbox{in}\ \  \RR^N \!\setminus\! \ol{B_{R_{0}}},\]
	where
	\begin{align}\label{eq4.6}
		\begin{aligned}
			l = |x|^{a} u^{-1} v^p \le C |x|^{\frac{a+pb}{p+1}} u^{\frac{pq-1}{p+1}}
			\le C |x|^{\frac{a+pb}{p+1} - \frac{(pq-1)\tilde{\alpha}}{p+1}}
			= C |x|^{-2},
			\quad \f x \in \RR^N \!\setminus\! \ol{B_{R_{0}}},
		\end{aligned}
	\end{align}
	for some $C=C(N,p,q,a,b,C_{0})>0,$ due to \eqref{eq2.2} and \eqref{eq4.5}. In this way, for any fixed $\delta \in (0,2),$ by Harnack's inequality (see \cite[p.\,255, Thm.\,B.4.2]{zbMATH05904617}) and \eqref{eq4.5}, we obtain
	\[ \|u\|_{L^{\oo}(B(y,|y|/4))}
	\le C_{1} |y|^{-\tilde{\alpha }} \|u\|_{L^{\frac{N}{\tl{\alpha }}}(B(y,|y|/2))}
	\le C_{1}\varepsilon  |y|^{-\tilde{\alpha }}, \]
	where $C_{1}>0$ only depends on $N,\tilde{\alpha}$ and $ |y|^{\delta}\|l\|_{L^{\frac{N}{2-\delta}}(B(y,|y|/2)}.$ But from \eqref{eq4.6}, one can see that
	\[ |y|^{\delta}\|l\|_{L^{\frac{N}{2-\delta}}(B(y,|y|/2))}
	\le C |y|^{\delta}  \| |x|^{-2} \|_{L^{\frac{N}{2-\delta}}(B(y,|y|/2))}
	\le C |y|^{\delta + \big( N - \frac{2N}{2-\delta}\big) \frac{2-\delta}{N}}
	= C,\]
	where $C>0$ depends only on $N,p,q,a,b,\delta,$ and $C_{0},$ hence so does $C_{1}.$ This completes the proof. \\
	(ii) \emph{Step 1. The estimate for $u.$} Firstly, for every $\varepsilon \in \left( 0,\frac{N-2}{2} \right),$ according to \eqref{eq2.2} and \eqref{eq4.2}, there exist $C_{0}=C_{0}(N,p,q,a,b)>0$ and $R_{\varepsilon}>0$ such that
	\[ |x|^{a} v^{p+1}(x) \le C_{0} |x|^{b} u^{q+1}(x), \quad
	u(x) \le C_{0}^{-\frac{p}{pq-1}} \varepsilon^{{\alpha}} |x|^{-\tilde{\alpha}},
	\quad \f\, |x| \ge R_{\varepsilon}.\]
	For this reason, we must have
	\[ -\Delta u = |x|^a v^p \le C_{0}^{\frac{p}{p+1}} |x|^{\frac{a+pb}{p+1}} u^{\frac{2}{\alpha}} u \le \varepsilon^2 |x|^{\frac{a+pb}{p+1}} |x|^{\frac{- 2\tilde{\alpha}}{\alpha}} u = \varepsilon^{2} |x|^{-2} u.  \]
	Let us define
	\[ m_{1} = \frac{N-2}{2} + \sqrt{\left( \frac{N-2}{2} \right)^{2}- \varepsilon^{2}},\quad m_{2} = \frac{N-2}{2} - \sqrt{\left( \frac{N-2}{2} \right)^{2}- \varepsilon^{2}},\]
	then $|x|^{-m_1},|x|^{-m_2} \in C^{\oo}(\RR^N \!\setminus\! \{0\})$ are both positive solutions of
	\[ -\Delta w - \varepsilon^{2}|x|^{-2} w =0 \quad \mbox{in}\ \  \RR^N \!\setminus\! \{0\}.\]
	For every $R>R_{\varepsilon},$ let
	\[ \tilde{u}(x) = R_{\varepsilon}^{m_1} \|u\|_{L^{\oo}(\RR^N)} |x|^{-m_1} + C_{0}^{-\frac{p}{pq-1}} \varepsilon^{{\alpha}} R^{m_2-\tilde{\alpha}} |x|^{-m_2}, \quad x \in \RR^N \!\setminus\! \{0\}. \]
	Consequently, we obtain
	\begin{align*}
		\left\{\begin{aligned}
			&(-\Delta-\varepsilon^{2} |x|^{-2}) (u-\tilde{u}) \le 0 &&  \mbox{in}\ \ B_{R} \!\setminus\! \ol{B_{R_{\varepsilon}}}, \\
			&u-\tilde{u} \le 0  &&   \mbox{on}\ \ \pt{\left( B_{R} \!\setminus\! \ol{B_{R_{\varepsilon}}} \right)}.
		\end{aligned}\right.
	\end{align*}
	By the weak maximum principle (see e.g. {\cite[p.\,48]{zbMATH06458598}}), we see that $u \le \tilde{u}$ in $B_{R} \!\setminus\! \ol{B_{R_{\varepsilon}}}. $ Taking $\varepsilon>0$ so small such that $m_2-\tilde{\alpha}<0$ and letting $R\ra +\oo,$ we get the estimate for $u.$ \\
	\emph{Step 2. The estimate for $v.$} For every $\varepsilon>0,$ by \eqref{eq2.2} and \eqref{eq4.2}, there exists $C>0$ such that
	\[v^p \le C |x|^{\frac{(b-a)p}{p+1}} u^{\frac{\alpha+2}{\alpha}} \le  C\varepsilon |x|^{\frac{(b-a)p}{p+1} - \frac{(\alpha+2)\tilde{\alpha}}{\alpha}} =  C\varepsilon |x|^{-p\tilde{\beta}}, \quad \f\, |x| \gg 1,\]
	which means $v=o(|x|^{-\tilde{\beta}})$ as $|x| \ra +\oo.$ Then for every sufficiently large $|x_{0}| > 1,$ let us define
	\[ U(y) = \left( \frac{|x_{0}|}{2} \right)^{\tilde{\alpha}} u\left( x_{0}+\frac{|x_{0}|}{2}y \right),\quad V(y) = \left( \frac{|x_{0}|}{2} \right)^{\tilde{\beta}} v\left( x_{0}+\frac{|x_{0}|}{2}y \right), \quad y \in B_{1}.\] Therefore $(U,V) \in \left( C^2(B_1) \right)^2$ is a solution of
	\begin{align}\label{eq3.23}
		\left\{\begin{aligned}
			-\Delta U	&= \left| \cdot + \frac{2x_{0}}{|x_{0}|} \right|^{a} V^p\\
			-\Delta V	&= \left| \cdot + \frac{2x_{0}}{|x_{0}|} \right|^{b} U^q
		\end{aligned}\right. \quad \mbox{in}\ \  B_1.
	\end{align}
	By the gradient estimates (see e.g. {\cite[p.\,41, Thm.\,3.9]{gilbargtrudinger1977}}), we conclude that
	\begin{align*}
		\begin{aligned}
			|DV(0)|	\le C \sup\limits_{B_{1}} V + C\sup\limits_{B_{1}} \left| \cdot + \frac{2x_{0}}{|x_{0}|} \right|^{b} U^q
			\le C\varepsilon |x_{0}|^{\tilde{\beta}-\tilde{\beta}} + C\varepsilon |x_{0}|^{q(\tilde{\alpha}-\tilde{\alpha})}
			\le C\varepsilon.
		\end{aligned}
	\end{align*}
	This gives $|Dv|=o(|x|^{-\tilde{\beta}-1})$ as $ |x| \ra +\oo.$ \\
	(iii) From \eqref{eq2.2}, \eqref{eq4.1} and \eqref{eq4.3}, we see that there exists $C>0$ such that
	\begin{align}\label{eq4.8}
		\begin{aligned}
			v^p \le C|x|^{-a} |x|^{\frac{a+pb}{p+1}} u^{\frac{\alpha+2}{\alpha}} \le C|x|^{-a} u^{-\frac{a+pb}{(p+1)\tilde{\alpha}}} u^{\frac{\alpha+2}{\alpha}}
			= C|x|^{-a} u^{\frac{\tilde{\alpha}+2}{\tilde{\alpha}}}	\le C|x|^{-a+\frac{(-N+2+\varepsilon)(\tilde{\alpha}+2)}{\tilde{\alpha}} } ,
		\end{aligned}
	\end{align}
	for every $|x| \ge 1$. Now using the gradient estimates for \eqref{eq3.23} again, combined with \eqref{eq4.3} and \eqref{eq4.8}, we have
	\begin{align*}
		\begin{aligned}
			|DU(0)|	&\le C \sup\limits_{B_{1}} U + C\sup\limits_{B_{1}} \left| y+ \frac{2x_{0}}{|x_{0}|} \right|^{a} V^p \\
			&\le C|x_{0}|^{-N+2+\varepsilon+\tilde{\alpha}} + C|x_{0}|^{p\tilde{\beta}  -a + \frac{(-N+2+\varepsilon)(\tilde{\alpha}+2)}{\tilde{\alpha}} } \\
			&= C|x_{0}|^{-N+2+\varepsilon+\tilde{\alpha}} + C|x_{0}|^{-N+2+\varepsilon+\tilde{\alpha} + \frac{2(-N+2+\varepsilon+\tilde{\alpha})}{\tilde{\alpha}} } .
		\end{aligned}
	\end{align*}
	Note that $0<\tilde{\alpha}<N-2,$ hence we can choose $\varepsilon>0$ so small such that $-N+2+\varepsilon+\tilde{\alpha}<0,$ and the conclusion follows easily. \\
	(iv) Suppose the assertion is false; then, according to Remark \ref{1.5}, \((u,v)\) must be positive in \(\mathbb{R}^N\). It is evident that
	\[ \frac{N+b}{q+1} < N-2,\quad \tilde{\alpha}>0,\quad \tilde{\beta}>0,\quad \tilde{\alpha}+\tilde{\beta}<N-2,\]
	and so $q>\frac{b+2}{N-2}$ and $\tilde{\alpha}<N-2.$ Then for every sufficiently small $\varepsilon>0$, using \eqref{eq2.2}, \eqref{eq4.3} and \eqref{eq4.4}, we conclude that
	\begin{align*}
		\begin{aligned}
			R^{b+1} \int_{S_R}	u^{q+1} &\le CR^{b+1+N-1+(q+1)(-N+2+\varepsilon)} = CR^{b+2-(N-2-\varepsilon)q + \varepsilon} \ra 0, \\
			R^{a+1} \int_{S_R}	v^{p+1} &= R\int_{S_R}|x|^{a}v^{p+1} \le CR\int_{S_R}	|x|^{b}u^{q+1} = CR^{b+1} \int_{S_R}	u^{q+1} \ra 0, \\
			\int_{S_R} |Du|v &\le CR^{N-1-N+1+\varepsilon-\tilde{\beta}} = CR^{-\tilde{\beta}+\varepsilon} \ra 0, \\
			\int_{S_R} |Dv|u &\le CR^{N-1-\tilde{\beta}-1-N+2+\varepsilon} = CR^{-\tilde{\beta}+\varepsilon} \ra 0, \\
			R\int_{S_R} |Du||Dv| &\le CR^{1+N-1-N+1+\varepsilon-\tilde{\beta}-1} = CR^{-\tilde{\beta}+\varepsilon} \ra 0,
		\end{aligned}
	\end{align*}
	as $R \ra +\oo.$ Then by Lemma \ref{2.12}, for every $c_1,c_{2} \in \RR$ with $c_1+c_2=N-2,$ we have
	\[ \lim\limits_{R \ra +\oo} \left[ \left(\dfrac{N+b}{q+1}-c_1\right) \int_{B_R}|x|^b u^{q+1}
	+\left(\dfrac{N+a}{p+1}-c_2\right) \int_{B_R}|x|^a v^{p+1} \right] = 0.\]
	By successively taking \(c_1 = \frac{N+b}{q+1}\) and \(c_1 = N-2 - \frac{N+a}{p+1}\), and noting that \((p,q,a,b)\) is supercritical, we arrive at contradictions.
\end{proof}

	As a consequence, Theorem \ref{1.10}\,(i) follows readily from Lemma \ref{2.3} and Theorem \ref{4.4}\,(iv). In addition, applying the result from \cite{zbMATH06121773} (as stated in the introduction) and {\cite[Thm.\,2]{zbMATH05563881}}, we establish Theorem \ref{1.10}\,(ii). Furthermore, utilizing \cite[Thm.\,1.1]{zbMATH07020411}, Theorem \ref{1.8}, Proposition \ref{3.5}, and Theorem \ref{1.10}\,(i), we derive Theorem \ref{1.10}\,(iii).

\appendix
\section{Appendix}\label{AA}

	In this section, we present several foundational results utilized in this paper. We commence with an elliptic regularity result. Although this result may be considered standard, it is difficult to locate in the literature; therefore, we provide complete proofs for clarity.

\begin{lemma}\label{2.1}
	Let $N \ge 3, \Omega \sub \RR^N$ be an open set and $K \sub \Omega$ be a compact set. Assume $k,l \in [1,+\oo]$ and $w \in W^{2,k}(\Omega).$\\
	(i) If $\frac{1}{k} - \frac{1}{l} < \frac{2}{N},$ then
	\[	
	\|w\|_{L^{l}(K)} \le C\left( \|\Delta w\|_{L^k(\Omega)} +  \|w\|_{L^1(\Omega)}\right),
	\]
	where $C=C(N,k,l,\Omega,K)>0.$ If in addition that $\Omega \sub B_{r}$ for some $r>0$ and $w \in W_{0}^{2,k}(\Omega),$ then
	\[	\|w\|_{L^{l}(K)} \le C r^{-\left( \frac{1}{k}-\frac{1}{l} \right)N + 2} \|\Delta w\|_{L^k(\Omega)} , \]
	where $C=C(N)>0$. \\
	(ii) If $\frac{1}{k} - \frac{1}{l} < \frac{1}{N},$ then
	\[
	\|D w\|_{L^l(K)} \le
	C\left(  \|\Delta w\|_{L^k(\Omega)}+ \| w \|_{L^1(\Omega)}\right),
	\]
	where $C=C(N,k,l,\Omega,K)>0.$ If in addition that $\Omega \sub B_{r}$ for some $r>0$ and $w \in W_{0}^{2,k}(\Omega),$ then
	\[
	\|D w\|_{L^l(K)} \le
	C r^{-\left( \frac{1}{k}-\frac{1}{l} \right)N + 1} \|\Delta w\|_{L^k(\Omega)} ,
	\]
	where $C=C(N)>0$.
\end{lemma}
\begin{proof}
	(i) \emph{Step 1.} Firstly, we can assume $w \in W^{2,k}(\Omega) \cap C^{\oo}(\Omega)$ by using an approximation argument. Without loss of generality we can also assume $\Omega \neq \RR^N.$ Then let us define
	\[R=\frac{1}{4}\operatorname{dist\,}(K,\partial\Omega) \in (0,+\oo).\]
	Since $K$ is compact, we can find $\{x_i\}_{i=1}^{j} \sub K $ such that $K \sub \cup_{i=1}^{j} B(x_i,R).$ Fix $x_i$ and $r \in \left[\frac{4}{3}R, \frac{5}{3}R\right] \!.$ We can assume, up to a rigid motion, that $x_i=0.$ Now Green's representative theorem implies
	\[w(x)= - \int_{B_r } G_r(x,y)\Delta w(y) \md y  -  \int_{S_r} w(y) D_y G_r(x,y) \cdot \nu(y) \md \sigma(y),\]
	where $x \in B_R  \Subset  B_{4R/3} \sub B_r
	$ and $G_r$ is the Green function of $-\Delta$ in $B_r.$ It is well known that
	\[|G_r(x,y)| \le \frac{C}{|x-y|^{N-2}}, \quad |D_y G_r(x,y)|\le \frac{C}{|x-y|^{N-1}} ,\quad \f x,y \in \ol{B_r},\]
	where $C=C(N)>0$ is independent of $r.$ For this fact, see e.g. {\cite[Thm.\,4.1]{hoff2025}} and note that
	\[G_r(x,y)=r^{2-N}G_{1}\left(\frac{x}{r},\frac{y}{r}\right) ,\quad \f x,y \in \ol{B_{r}}.\]
	Therefore, for $x \in B_R,$ we must have
	\begin{align*}
		|w(x)| &\le C\int_{B_{r} } \frac{|\Delta w(y)|}{|x-y|^{N-2}} \md y  +  C\int_{S_{r} } |w| \md \sigma \\
		&\le C(f*g)(x) +C\int_{S_{r}} w \md \sigma,
	\end{align*}
	where
	\begin{align*}
		f &: \RR^N \ra \RR ,\, z \mapsto \left\{\begin{aligned}
			&\frac{1}{|z|^{N-2}} ,&& z \in B_{\frac{8}{3}R} ,\\
			&0 ,&& \mbox{others},
		\end{aligned}\right.	\\
		g &: \RR^N \ra \RR ,\, z \mapsto \left\{\begin{aligned}
			&|\Delta w(z)| ,&& z \in \Omega,\\
			&0 ,&& \mbox{others},
		\end{aligned}\right.
	\end{align*}
	and $C=C(N,K,\Omega)>0.$ Thus by Young's inequality, we see that
	\begin{align*}
		\|w\|_{L^l(B_R) } &\le C\|f*g\|_{L^l(B_R)} + C\int_{S_{r}} |w| \md \sigma  \\
		&\le C\|f*g\|_{L^l(\RR^N)} + C\int_{S_{r} } |w| \md \sigma \\
		&\le C\|f\|_{L^m(\RR^N)}\|g\|_{L^k(\RR^N)} + C\int_{S_{r} } |w| \md \sigma \\
		&\le C\|f\|_{L^m(\RR^N)}\|\Delta w\|_{L^k(\Omega)} + C\int_{S_{r}} |w| \md \sigma
	\end{align*}
	where $\frac{1}{m} = 1+\frac{1}{l} - \frac{1}{k}$ and $ C=C(N,\Omega,K)>0.$ Since $ \frac{1}{k} - \frac{1}{l} < \frac{2}{N},$ we obtain $f \in L^m(\RR^N)$ and $C\|f\|_{L^m(\RR^N)} = C(N,k,l,\Omega,K)>0.$ Then, integrating over $r \in \left[\frac{4}{3}R, \frac{5}{3}R\right] \!,$ we get
	\[\|w\|_{L^l(B_R) } \le C\left(\|\Delta w\|_{L^k(\Omega)} + \|w\|_{L^1(\Omega)}\right).\]
	Finally, we have
	\[\|w\|_{L^l(K) } \le \sum\limits_{i=1}^{j}\|w\|_{L^l(B(x_i,R)) } \le C\left(\|\Delta w\|_{L^k(\Omega)} + \|w\|_{L^1(\Omega)}\right).\]
	\emph{Step 2.} If in addition that $w \in W_{0}^{2,k}(\Omega),$ then by an approximation argument, we can also assume $w \in C_c^{\oo}(\Omega).$ Now we can use Green's representative theorem in $B_{r}$ and proceed analogously to Step 1.
	\\
	(ii) The proof of this conclusion is similar to that of (i). We just need to note that
	\[|D_x G_r(x,y)| \le \frac{C}{|x-y|^{N-1}}, \quad |D_x D_y G_r(x,y)|\le \frac{C}{|x-y|^{N}} ,\quad \f x,y \in \ol{B_{r}} ,\]
	where $C=C(N)>0$ is independent of $r.$ For this fact, see e.g. {\cite[Thm.\,4.1]{hoff2025}} and note that
	\[G_r(x,y)=r^{2-N}G_{1}\left(\frac{x}{r},\frac{y}{r}\right) ,\quad \f x,y \in \ol{B_{r}}. \qedhere \]
\end{proof}

	We also require the following basic energy estimates for nonnegative solutions of \eqref{HLE}; see {\cite{zbMATH00883497}}, {\cite{zbMATH06265786}}, and {\cite{zbMATH06973892}}. For convenience, we provide a concise and novel proof of this result, which extends to more general systems.

\begin{lemma}\label{2.2}
	Let $N \ge 3,p,q>0,pq > 1$ and $a,b>-2$. Additionally, assume $(u,v)$ is a nonnegative solution of \eqref{HLE}. \\
	(i) For every $R>0,$ we have
	\begin{equation*}
		\int_{B_R} |x|^b u^q \le CR^{N-2-\tilde{\beta}}, \quad \int_{B_R} |x|^a v^p \le CR^{N-2-\tilde{\alpha}},
	\end{equation*}
	where $C=C(N,p,q,a,b)>0.$ \\
	(ii) If $\gamma \in (0,1],$ then
	\[
	\ol{u^\gamma}(R) \leq C R^{-\gamma\tilde{\alpha}},\quad
	\ol{v^\gamma}(R) \leq C R^{-\gamma\tilde{\beta}},
	\]
	where $C=C(N,p,q,a,b)>0.$ \\
	(iii) If $\gamma \in (0,1)$ and $R_{3}>R_{2}>0$, then
	\[ \int_{B_{R_{3}} \setminus B_{R_{2}}}  |D(u^{\frac{\gamma}{2}})|^{2} + |D(v^{\frac{\gamma}{2}})|^{2}  \le \frac{\gamma }{1-\gamma } C , \]
	where $C=C(N,p,q,a,b,R_{2},R_{3})>0.$
\end{lemma}

	Our proof relies on the following result.

\begin{lemma}[A quantitative form of the Hopf lemma]\label{2.6}
	Let $N\ge 3 $ and $\Omega \subseteq \RR^N $ be a bounded connected open set of class $C^{2,\gamma}$ for some $\gamma \in (0,1).$ Assume that $w \in  W^{2,1}(\Omega)  $ satisfies
	\[ \left\{\begin{aligned}
		-\Delta w \ge 0& \quad \mbox{a.e. in \,} \Omega,\\
		w \ge 0& \quad \mbox{a.e. on \,} \partial \Omega.
	\end{aligned}\right. \]
	(i) There holds
	\[ w \ge C(\Omega) \!\left[\int_{\Omega} (-\Delta w) d\right]\! d \quad \text{a.e. in}\ \ \Omega,\]
	where $C(\Omega)>0$ and $d=\operatorname{dist\,}(\cdot,\partial\Omega).$\\
	(ii) For any compact set $K \sub \Omega,$ we have
	\[\operatorname{ess }\inf_{K} w \ge C(\Omega)\operatorname{dist\,}(K,\partial\Omega)^2\int_{K} (-\Delta w), \]
	where $C(\Omega)>0.$
\end{lemma}
\begin{proof}
	(i) Let $G: \ol{\Omega} \times \ol{\Omega} \ra [0,+\oo]$ be the Green function of $-\Delta$ in $\Omega,$ then by {\cite[Thm.\,1]{zbMATH03983703}}, we have
	\[G(x,y) \ge C(\Omega)d(x)d(y),\quad \f (x,y) \in \ol{\Omega} \times \ol{\Omega}.\]
	For every $x \in \Omega,$ since $\left. G(x,\cdot) \right|_{\Omega}>0$ and $\left. G(x,\cdot) \right|_{\pt\Omega}=0$, it's clear that $D_y G(x,y) \cdot \nu(y) \le 0$ for every $y \in \partial\Omega.$ Hence Green's representative theorem implies that, for a.e. $x \in \Omega,$ we have
	\begin{align*}
		w(x) &= \int_{\Omega} \left[-\Delta w(y)\right] G(x,y) \md y + \int_{\partial\Omega} w(y)\left[-D_y G(x,y) \cdot \nu(y)\right] \mathrm{d} \sigma(y) \\
		&\ge C(\Omega) \left[\int_{\Omega} (-\Delta w) d\right] d(x).
	\end{align*}
	(ii) This conclusion can be easily derived from (i). Indeed, for a.e. $x\in K,$ we have
	\[w(x)  \ge C(\Omega) d\,(K,\partial\Omega) \int_{K} (-\Delta w) d \ge  C(\Omega) d\,(K,\partial\Omega)^2 \int_{K} (-\Delta w). \qedhere \]
\end{proof}
\begin{remark}\label{2.7}
	Note that if $N\ge 3,p,q>0,a, b>-2$ and $(u,v)$ is a nonnegative solution of \eqref{HLE}, then we must have $u,v \in W_{\rm{loc}}^{2,\frac{N}{2}}(\RR^N) $ due to the elliptic regularity theory. Consequently, one can use Lemma \ref{2.6} for $u$ or $v$ in any open ball or annulus.
\end{remark}

\begin{proof}[\bf{Proof of Lemma \ref{2.2}}]
	(i) By a rescaling argument, it suffices to prove the inequalities when $R=1.$ By Lemma \ref{2.6}\,(ii) and Remark \ref{2.7}, we have
	\[
	\inf_{B_1} u \ge C\int_{B_1} (-\Delta u) \ge C\int_{B_1} |x|^a v^p(x) \md x,
	\]
	where $C=C(N)>0.$ Consequently, we obtain
	\begin{equation}\label{eq4.10}
		\int_{B_1} |x|^b u^q(x) \md x\ge C\int_{B_1} |x|^b \left(\int_{B_1} |y|^a v^p(y) \md y\right)^q  \mathrm{d}x \ge C\left(\int_{B_1} |y|^a v^p(y) \md y\right)^q,
	\end{equation}
	where $C=C(N,q,b)>0.$ Similarly, we have
	\begin{equation}\label{eq4.11}
		\int_{B_1} |y|^a v^p(y) \md y  \ge C\left(\int_{B_1} |x|^b u^q(x) \md x\right)^p,
	\end{equation}
	where $C=C(N,p,a)>0.$ Then combining \eqref{eq4.10} and \eqref{eq4.11}, we complete the proof. \\
	(ii) \emph{Step 1.} It suffices to prove that
	\begin{align}\label{eq2.8}
		\begin{aligned}
			\ol{u}(R) \leq C R^{-\tilde{\alpha}},\quad
			\ol{v}(R) \leq C R^{-\tilde{\beta}},
		\end{aligned}
	\end{align}
	due to Jensen's inequality. A straightforward calculation shows that $ \ol{u^{q_{1}}}$ and $\ol{v^{p_{1}}}$ are decreasing on $[0,+\oo)$ for every $p_{1},q_{1} \in (0,1]$. Hence if $\min\, \{p,q\} \ge 1,$ then \eqref{eq2.8} can be easily deduced from Lemma \ref{2.2}, H\"{o}lder's inequality and the coarea formula. What is left is to show that \eqref{eq2.8} holds for $\min\, \{p,q\}<1.$ Since $pq>1,$ without restriction of generality we can assume that $0<q<1<p.$ In this case, by Lemma \ref{2.2}, H\"{o}lder's inequality and the coarea formula, we have
	\begin{align}\label{eq2.9}
		\begin{aligned}
			\ol{u^{q}}(r) \leq C r^{-q\tilde{\alpha}},\quad
			\ol{v}(r) \leq C r^{-\tilde{\beta}},\quad  \f r>0.
		\end{aligned}
	\end{align}
	\emph{Step 2.} We claim that if there exist $s \in (0,1)$ and $C=C(N,p,q,a,b)>0$ such that
	\[ \ol{u^{s}}(r) \leq C r^{-s\tilde{\alpha}}, \quad \f r >0, \]
	then there exists $C=C(N,p,q,a,b)>0$ such that
	\[ \ol{u^{\frac{Ns}{N-2}}}(r) \leq \left( \frac{s}{1-s} \right)^{\frac{N}{N-2}} C r^{-\frac{Ns}{N-2} \tilde{\alpha}}, \quad \f r >0, \]
	Indeed, note that
	\[-\Delta (u^s) = su^{s-1}(-\Delta u) + \frac{4(1-s)}{s}|D(u^{\frac{s}{2}})|^2 \ge \frac{4(1-s)}{s}|D(u^{\frac{s}{2}})|^{2} \ge 0,\]
	by Lemma \ref{2.6}\,(ii), we must have
	\begin{align}\label{eq2.10}
		\begin{aligned}
			\int_{B_4 \setminus B_{1/2}} |D(u^{\frac{s}{2}})|^2 \le \frac{sC}{1-s} \inf_{B_{4} \setminus B_{1/2}} u^s \le \frac{sC}{1-s} \ol{u^s}(1) \le \frac{sC}{1-s},
		\end{aligned}
	\end{align}
	where $C=C(N,p,q,a,b)>0.$ Fix $\eta \in C_{c}^{\oo}(B_{4} \!\setminus\! B_{1/2})$ such that $\left. \eta \right|_{B_{2} \setminus B_{1}}=1.$
	Since
	\[ \|D(u^{\frac{s}{2}} \eta )\|_{L^2(B_4 \setminus B_{1/2})} \le C\left[ \|u^{\frac{s}{2}}\|_{L^2(B_4 \setminus B_{1/2})} + \|D(u^{\frac{s}{2}})\|_{L^2(B_4 \setminus B_{1/2})} \right] \le  \left(\frac{s}{1-s}\right)^{\frac{1}{2}} C, \]
	then one can use Sobolev's inequality to deduce that
	\[ \|u^{\frac{s}{2}} \|_{L^{\frac{2N}{N-2}}(B_2 \setminus B_1)}
	\le \|u^{\frac{s}{2}} \eta \|_{L^{\frac{2N}{N-2}}(B_4 \setminus B_{1/2})}
	\le C\|D(u^{\frac{s}{2}}\eta) \|_{L^2(B_2 \setminus B_1)}
	\le \left(\frac{s}{1-s}\right)^{\frac{1}{2}} C.\]
	For this reason, there exists $r_{0} \in (1,2)$ such that
	\[ \ol{u^{s}}(r_{0}) \leq \left( \frac{s}{1-s} \right)^{\frac{N}{N-2}} C . \]
	Therefore, we can use the monotonicity of $\ol{u^s}$ and a scaling argument to complete the proof of the claim. \\
	\emph{Step 3.} Taking $m \in \NN_+$ such that $\big(\frac{N}{N-2}\big)^{m-1}q \le 1$ and $q_m := \big(\frac{N}{N-2}\big)^{m}q > 1$, then using \eqref{eq2.9} and the above claim to iterate, we obtain
	\[ \ol{u^{q_{m}}}(r) \le \left(\frac{q}{1-q}\right)^{\frac{Nm}{N-2}} C r^{-q_{m}\tilde{\alpha}} \le C(N,p,q,a,b) r^{-q_{m}\tilde{\alpha}} ,\quad \f r >0,\]
	which, combined with Jensen's inequality, complete the proof. \\
	(iii) This follows directly from (ii) and \eqref{eq2.10}.
\end{proof}

	We conclude this section with a comparison property established by Q. Phan {\cite[Lem.\,2.7]{zbMATH06121773}}. As noted in that work, the proof builds on the ideas presented in \cite[Thm.\,3.1]{zbMATH01700716} and \cite[Lem.\,2.7]{zbMATH05563881}. While the original result assumed $N \ge 3$, we demonstrate that this comparison property actually holds for $N \ge 2$. In contrast to the original approach, the following proof, inspired by the rescaled test function method, does not require any energy estimate or further property of solutions.

\begin{lemma}[Comparison property]\label{2.3}
	Suppose $N \ge 2,$
	$p \ge q>0,$ $pq>1,a,b>-2$
	and $(u, v)$ is a positive solution of \eqref{HLE}. If \eqref{eq1.3} is satisfied, then
	\begin{align}\label{eq2.2}
		\begin{aligned}
			|x|^a v^{p+1} \le  C(N,p,q,a,b) |x|^b u^{q+1}, \quad \f x \in \RR^N \!\setminus\! \{0\}.
		\end{aligned}
	\end{align}
	for some $C(N,p,q,a,b)>0.$ Moreover, we have $C(N,p,q,a,b)= \frac{p+1}{q+1}.$
\end{lemma}
\begin{proof}
	\emph{Step 1.} Let us first define \[ \sigma = \frac{q+1}{p+1} \in (0,1],\quad
	l=\sigma^{-\frac{1}{p+1}},\quad
	h=\frac{a-b}{p+1},\quad  w = v - l |x|^{-h} u^{\sigma}.\] The task is now to prove that $w \le 0$ in $\RR^N \!\setminus\! \{0\}. $ First, an easy computation shows that
	\begin{align}\label{eq2.11}
		\begin{aligned}
			\Delta w = |x|^{a-h}u^{\sigma-1} \left[  \left(\frac{v}{l}\right)^p  -  \left(|x|^{-h} u^{\sigma} \right)^p  \right]
			+ l|x|^{-h}u^{\sigma} M,
		\end{aligned}
	\end{align}
	for every $x \in \RR^N \!\setminus\! \{0\}$, where
	\[M= \sigma(1-\sigma) \frac{|Du|^2}{u^2}  +  2\sigma h \frac{x}{|x|^2} \cdot \frac{Du}{u}  +  \frac{(N-2-h)h}{|x|^2}.\]
	If $\sigma=1,$ then $h=0,$ thus $M=0$; If $\sigma \in (0,1),$ then it is obvious that $h \ge 0 $ and $N-2-\frac{h}{1-\sigma} \ge 0$ due to \eqref{eq1.3}, hence
	\[M=\sigma(1-\sigma) \left|\frac{Du}{u} + \frac{h}{1-\sigma}\cdot \frac{x}{|x|^2} \right|^2  +  h\left(N-2-\frac{h}{1-\sigma}\right)\frac{1}{|x|^2} \ge 0.\]
	Therefore we have $M \ge 0$ in either case. Now since
	\[ \frac{x^{p}-1}{(x+1)^{p-1}(x-1)} \le C,\quad \f x \in (0,1), \]
	and
	\[ \frac{1}{C} \le \frac{x^{p}-1}{(x+1)^{p-1}(x-1)} ,\quad \f x > 1,\]
	for some $C=C(p)>0$ large enough, we deduce that
	\[ s^p - t^p \ge C (s+t)^{p-1} (s-t),\quad \f s , t \ge 0.\]
	Inserting this inequality into \eqref{eq2.11} and note $M\ge 0,$ we see that
	\begin{align}\label{eq2.5}
		\begin{aligned}
			\Delta w &\ge C |x|^{a-h} u^{\sigma-1} \left( \frac{v}{l} + |x|^{-h}u^{\sigma} \right)^{p-1} \left( \frac{v}{l} - |x|^{-h}u^{\sigma} \right)  \\
			&=C |x|^{a-h} u^{\sigma-1} (w + 2l|x|^{-h} u^{\sigma})^{p-1} w \\
			&\ge C \left( |x|^{a-h} u^{\sigma-1} w^{p} + |x|^{a-ph} u^{p\sigma-1} w \right),
	\end{aligned}\end{align}
	for every $x \in \RR^N \!\setminus\! \{0\},$ where $C=C(p,q)>0$. \\
	\emph{Step 2.} Let $\eta \in C_c^{\oo}(B_1)$ such that $0 \le \eta \le 1$ and $\left.\eta\right|_{B_{1/2}}=1.$ Then for every $R \in (0,+\oo),$ let us define \[m=\frac{2\sigma}{p\sigma-1}>0, \quad
	\eta_{R} = \eta^m\left(\frac{\cdot}{R}\right), \quad
	w_R=w\eta_R.\]
	Next, we claim that for every $R>0,$ there exists $x_R \in {B_R} \!\setminus\! \{0\}$ such that $w_R(x_R) = \sup_{B_R \setminus \{0\}} w_R .$
	In fact, if $a>b,$ it suffices to note that $w_R \in C(\ol{B_R} \!\setminus\! \{0\}), \left. w_R \right|_{S_R}=0$ and $\lim\limits_{x \ra 0} w_R(x) = -\oo ;$ If $a=b$ and the claim is false, since $w \in C(\ol{B_R})$ in this case, then $w(0)>0$ and $w(0)>w(x)$ for every $x \in \RR^N \!\setminus\! \{0\}.$ It follows that there exists $r>0$ such that $\left.w\right|_{B_r}>0,$ which forces $\left. \Delta w \right|_{B_r \setminus \{0\}}>0$ by \eqref{eq2.5}. This contradicts the maximum principle in punctured balls (see, for example, {\cite[Thm.\,1]{zbMATH03118705}}). We complete the proof of the claim.	 \\
	\emph{Step 3.} Suppose the assertion of the lemma is false, then $\sup_{\RR^N \setminus \{0\}} w>0.$ Thus $w_R(x_R)>0$ and $\eta_{R}(x_R)>0$ for every sufficiently large $R>0.$ Now we have
	\[w(x_R) D\eta_R (x_R) + \eta_R(x_R) Dw(x_R) = Dw_R(x_R)=0, \]
	and
	\[w(x_R) \Delta\eta_R(x_R) +2 Dw(x_R) D\eta_R(x_R) + \eta_R(x_R) \Delta w(x_R) = \Delta w_R(x_R) \le 0.\]
	From this and note that
	\[|D\eta_R| \le CmR^{-1}\eta_{R}^{1-\frac{1}{m}}, \quad
	|\Delta \eta_R| \le CmR^{-2}\left[\eta_{R}^{1-\frac{1}{m}} + (m-1) \eta_{R}^{1-\frac{2}{m}}\right], \]
	where $C=C(N)>0,$ we can easily deduce that
	\begin{align}\label{eq2.6}
		\begin{aligned}
			\Delta w(x_R) &\le \left( \left|w(x_R)\Delta \eta_R(x_R)\right| + 2\left| Dw(x_R) D\eta_R(x_R)\right| \right) \eta_{R}^{-1}(x_R) \\
			& \le \left( \left|\Delta \eta_R (x_R)\right| \eta_R^{-1}(x_R)
			+ 2\left|D \eta_R (x_R)\right|^2 \eta_R^{-2}(x_R)\right) w(x_R) \\
			&\le CR^{-2} \eta_{R}^{-\frac{2}{m}}(x_R) w(x_R),
	\end{aligned}\end{align}
	where $C=C(N,p,q)>0$ and $R>0$ is sufficiently large. \\
	\emph{Step 4.} From \eqref{eq2.5} and \eqref{eq2.6}, we get
	\[|x_R|^{a-h} u^{\sigma-1}(x_R) w^p(x_R) + |x_R|^{a-ph} u^{p\sigma-1}(x_R) w(x_R) \le CR^{-2} \eta_{R}^{-\frac{2}{m}}(x_R) w(x_R) .\]
	It follows that
	\[w_{R}^{p-1}(x_R) \le CR^{-2} |x_R|^{h-a} \eta_{R}^{p-1-\frac{2}{m}}(x_R) u^{1-\sigma}(x_R), \]
	and
	\[u(x_R) \le CR^{-\frac{2}{p\sigma-1}} |x_R|^{\frac{ph-a}{p\sigma-1}} \eta_R^{-\frac{2}{(p\sigma-1)m}}(x_R). \]
	Consequently, we have
	\begin{align*}
		w_{R}^{p-1}(x_R) &\le CR^{-2-\frac{2(1-\sigma)}{p\sigma-1}}|x_R|^{h-a+\frac{(ph-a)(1-\sigma)}{p\sigma-1}}\eta_{R}^{p-1-\frac{2}{m}-\frac{2(1-\sigma)}{(p\sigma-1)m}} \\
		&= CR^{-2-\frac{2(1-\sigma)}{p\sigma-1}}|x_R|^{h-a+\frac{(ph-a)(1-\sigma)}{p\sigma-1}} \\
		&\le CR^{-2-\frac{2(1-\sigma)}{p\sigma-1} + h-a+\frac{(ph-a)(1-\sigma)}{p\sigma-1} },
	\end{align*}
	where $C=C(N,p,q)>0$ and $R>0$ is sufficiently large.
	To drive a contradiction, we only need to show that
	\[-2-\frac{2(1-\sigma)}{p\sigma-1} + h-a+\frac{(ph-a)(1-\sigma)}{p\sigma-1} <0. \]
	This is equivalent to
	\[\frac{\left[(a+2)q+b+2\right](p-1)}{pq-1}>0,\]
	which is obviously true.
\end{proof}

\newpage

\printbibliography[heading=bibintoc, title=\ebibname]
\vspace{1em}
\begin{flushleft} \small
	Long-Han Huang: Department of Mathematical Sciences, Tsinghua University. No.\,1 Tsinghua Garden, Beijing 100084, People's Republic of China. \\
	E-mail address: \href{mailto:huanglh22@mails.tsinghua.edu.cn}{\texttt{huanglh22@mails.tsinghua.edu.cn}}
\end{flushleft}

\begin{flushleft} \small
	Wenming Zou: Department of Mathematical Sciences, Tsinghua University. No.\,1 Tsinghua Garden, Beijing 100084, People's Republic of China. \\
	E-mail address: \href{mailto:zou-wm@mail.tsinghua.edu.cn}{\texttt{zou-wm@mail.tsinghua.edu.cn}}
\end{flushleft}


\end{document}